\documentclass{amsart}

\usepackage[svgnames]{xcolor}

\usepackage{tikz-cd}

\usepackage{euscript}
\usepackage{combelow}

\usepackage{esint}
\usepackage[abbrev]{amsrefs}
\usepackage{amssymb}
\usepackage{algpseudocode}
\usepackage{algorithm}
\usepackage{mathrsfs}
\usepackage{accents}
\usepackage{enumitem}

\usepackage[hidelinks]{hyperref}
\usepackage[nameinlink]{cleveref}

\makeatletter
\let\div\@undefined
\makeatother
\DeclareMathOperator{\div}{div}
\DeclareMathOperator{\diam}{diam}

\DeclareMathOperator{\Lip}{Lip}
\DeclareMathOperator{\Lipb}{Lip_b}

\newcommand{\cointerval}[1]{[#1)}

\newcommand{\abs}[1]{\left\vert#1\right\vert}

\newcommand{\norm}[1]{\Vert#1\Vert}
\newcommand{\set}[1]{\left\{#1\right\}}
\newcommand{\Morrey}[1]{\norm{#1}_{\mathcal{M}^1}}
\newcommand{\Xden}{\mathscr{C}}
\newcommand{\Xgeo}{\mathscr{G}}
\newcommand{\Xlsi}{\mathscr{L}}
\newcommand{\nsmallpieces}{m}
\newcommand{\length}{\ell}

\makeatletter

\newenvironment{breakablealgorithm}
  {
   \begin{center}
     \refstepcounter{algorithm}
     \hrule height.8pt depth0pt \kern2pt
     \renewcommand{\caption}[2][\relax]{
       {\raggedright\textbf{\fname@algorithm~\thealgorithm} ##2\par}%
       \ifx\relax##1\relax 
         \addcontentsline{loa}{algorithm}{\protect\numberline{\thealgorithm}##2}%
       \else 
         \addcontentsline{loa}{algorithm}{\protect\numberline{\thealgorithm}##1}%
       \fi
       \kern2pt\hrule\kern2pt
     }
  }{
     \kern2pt\hrule\relax
   \end{center}
  }
\makeatother

\newtheorem{theorem}{Theorem}[section]
\newtheorem{lemma}[theorem]{Lemma}
\newtheorem{corollary}[theorem]{Corollary}

\theoremstyle{definition}
\newtheorem{definition}[theorem]{Definition}

\theoremstyle{remark}
\newtheorem{remark}[theorem]{Remark}

\numberwithin{equation}{section}

\let\epsilon\varepsilon 

\begin{document}

\title[Atomic decomposition of metric currents]{An atomic decomposition of one-dimensional metric currents without boundary}


\author{}
\address{}
\curraddr{}
\email{}
\thanks{}

\author{}
\address{}
\curraddr{}
\email{}
\thanks{}

\author[Y.-W.~B. Chen]{You-Wei Benson Chen}
\address[Y.-W.~B. Chen]{Department of Mathematics, National Taiwan University, Taipei 10617, R.O.C.}
\email[Y.-W.~B. Chen]{bensonchen@ntu.edu.tw}
\thanks{} 

\author[J. Goodman]{Jesse Goodman}
\address[J. Goodman]{Department of Statistics, University of Auckland, Private Bag 92019, Auckland 1142, New Zealand}
\email[J. Goodman]{jesse.goodman@auckland.ac.nz}
\thanks{} 

\author[F. Hernandez]{Felipe Hernandez}
\address[F. Hernandez]{Department of Mathematics, MIT, 182 Memorial Dr, Cambridge, MA 02139, USA}
\email[F. Hernandez]{felipeh@mit.edu}

\author[D. Spector]{Daniel Spector}
\address[D. Spector]{Department of Mathematics, National Taiwan Normal University, No. 88, Section 4, Tingzhou Road, Wenshan District, Taipei City, Taiwan 116, R.O.C.
\newline
and
\newline
National Center for Theoretical Sciences\\No. 1 Sec. 4 Roosevelt Rd., National Taiwan
University\\Taipei, 106, Taiwan
\newline
and
\newline
Department of Mathematics, University of Pittsburgh, Pittsburgh, PA 15261 USA
}
\email[D. Spector]{spectda@gapps.ntnu.edu.tw}


\date{}

\dedicatory{}

\commby{}

\begin{abstract}
  This paper proves an atomic decomposition of the space of $1$-dimensional metric currents without boundary, in which the atoms are specified by closed Lipschitz curves with uniform control on their Morrey norms.
Our argument relies on a geometric construction which states that for any $\epsilon>0$ one can express a piecewise-geodesic closed curve as the sum of piecewise-geodesic closed curves whose total length is at most $(1+\epsilon)$ times the original length and whose Morrey norms are each bounded by a universal constant times $\epsilon^{-2}$.  
In Euclidean space, our results refine the state of the art, providing an approximation of divergence free measures by limits of sums of closed polygonal paths whose total length can be made arbitrarily close to the norm of the approximated measure.
\end{abstract}

\maketitle

\section{Introduction}

Function spaces emerge readily in the modeling of physical phenomena via the calculus of variations and partial differential equations.  
A question of fundamental importance in understanding these equations, energies, and the underlying phenomena is then the analysis of these function spaces into their basic building blocks -- atoms -- as well as the synthesis of general elements of these spaces -- atomic decompositions.  
Typically, an atomic decomposition is built around three basic conditions:
\begin{enumerate}
\item Support Condition -- an atom is zero outside a set from a standard class;
\item Size Condition -- some norm of an atom is bounded in terms of its support;
\item Cancellation Condition -- an atom integrates to zero against a specified class of test functions.
\end{enumerate}
This paradigm was first employed for Hardy spaces in \cites{Coifman, GarnettLatter,Latter,LatterUchiyama}, and has been extended to other function spaces, e.g.\ Besov and Tribel-Lizorkin spaces \cites{AH, Grafakos}.  
As noted by Coifman and Weiss \cite{CoifmanWeiss}*{p.~574}, atomic decompositions often simplify existing proofs, allow easy proofs of new results, and enable the definition of these spaces to be extended to new situations with theories as useful and powerful as the original ones.  

In this paper we establish an atomic decomposition for the space of $1$-dimensional metric currents without boundary.   
The atoms in our setting are specified by closed Lipschitz curves, which can be canonically mapped into the space of $1$-dimensional metric currents, which we abbreviate as $1$-currents in the sequel.  
The support condition reflects the fact that the resulting $1$-current is supported in the image of the curve.  
The size condition involves a bound on the atoms in a Morrey space.  
The cancellation condition reflects the fact that a closed Lipschitz curve maps to a $1$-current without boundary.  

In the Euclidean setting, $1$-currents can be identified with vector-valued measures on $\mathbb{R}^n$, acting on vector-valued functions $\omega \in C_c(\mathbb{R}^n;\mathbb{R}^n)$ via
\begin{align*}
  T(\omega) = \int_{\mathbb{R}^n} \omega\cdot dT
  .
\end{align*}
In this identification, $1$-currents without boundary correspond to divergence free vector-valued measures, because
\begin{align*}
\partial T(\phi) := T(d\phi) = \int_{\mathbb{R}^n} \nabla \phi \cdot dT= -\langle\operatorname*{div}T,\phi\rangle
\end{align*}
for all $\phi \in C^1_c(\mathbb{R}^n)$.  
The analogue of our atomic decomposition in this setting was established in \cite{HS}.

\subsection{Main result}\label{MainResultSubsect}

We begin by recalling some notation which is further defined in \Cref{differentialforms}.
For $E$ a complete, separable, geodesic metric space, we denote by $\mathcal{D}^1(E),\mathcal{M}_1(E)$ the space of $1$-forms and $1$-dimensional metric currents on $E$, respectively, in the sense of Ambrosio and Kirchheim \cite{AmbrosioKirchheim}.  
For a 1-current $T \in \mathcal{M}_1(E)$, we denote by $\partial T$ its boundary and $\mathbb{M}(T)$ its mass.

Given a Lipschitz curve $\gamma:[0,l]\to E$, we denote by $[[\gamma]]$ the $1$-dimensional metric current associated to this curve.
Note that if $\gamma,\tilde{\gamma}$ are the same curve traversed in opposite directions then $[[\gamma]]+[[\tilde{\gamma}]]=0$.
We write $\length(\gamma)$ for the length of $\gamma$.
For a curve $\gamma\colon [0,\length(\gamma)]\to E$ parametrized by arc length, we write
\begin{equation}\label{BallGrowthConstant}
  \Morrey{\gamma}= \sup_{r>0, x\in E} \frac{|\set{t\in[0,\length(\gamma)]\colon d(x,\gamma(t))\leq r}|}{r}
\end{equation}
for the Morrey norm of the curve.  

We require in addition that $E$ should be a geodesic metric space in the sense that for every pair of points $x,y\in E$, there exists a length-minimizing geodesic with endpoints $x,y$.

Our main result gives an atomic decomposition for a metric current in terms of collections of curves.
\begin{theorem}\label{approximationMetric}
  There is a universal constant $C$ such that if $E$ is a complete, separable, geodesic metric space, $T \in \mathcal{M}_1(E)$ satisfies $\partial T=0$, and $0<\epsilon<1$, then there exist piecewise-geodesic closed curves $\gamma_{i,n}$ and scalars $\lambda_{i,n} \geq 0$, $n \in \mathbb{N}$, $1\leq i\leq n$, such that 
  \begin{gather}
    T(\omega)= \lim_{n \to \infty} \sum_{i=1}^n \lambda_{i,n} \frac{[[\gamma_{i,n}]](\omega)}{\length(\gamma_{i,n})} 
    \quad\text{for all $\omega\in \mathcal{D}^1(E)$}, \label{limit}
    \\
    \sum_{i=1}^n\abs{\lambda_{i,n}} \leq (1+\epsilon) \mathbb{M}(T)
    \quad\text{for all $n$, and} \label{mass}
    \\
    \Morrey{\gamma_{i,n}} \leq \frac{C}{\epsilon^2}
    \quad\text{for all $i,n$}. \label{morrey_bound}
  \end{gather}
\end{theorem}

We note that $\mathbb{M}([[\gamma]])\leq \length(\gamma)$, see \Cref{easyestimateformass}, so that the left-hand side of \eqref{mass} is an upper bound for the mass of the sum in \eqref{limit} and therefore
\begin{equation*}
  \mathbb{M}(T) \leq \liminf_{n \to \infty} \sum_{i=1}^n\abs{\lambda_{i,n}} \leq \limsup_{n \to \infty} \sum_{i=1}^n\abs{\lambda_{i,n}} \leq (1+\epsilon) \mathbb{M}(T) 
  .
\end{equation*}

\subsection{Discussion}
\subsubsection{The atoms of the decomposition}
If we interpret \Cref{approximationMetric} in the atom-style formulation, the atoms are 
\begin{equation*}
  a_{i,n}= \frac{[[\gamma_{i,n}]]}{\length(\gamma_{i,n})}
  .
\end{equation*}
These atoms $a_{i,n} \in \mathcal{M}_1(E)$ satisfy forms of the classical support, size, and cancellation conditions:
\begin{enumerate}[ref=(\arabic*)]
  \item\label{SupportCond}
  The support of $a_{i,n}$ is the image of the curve $\gamma_{i,n}$.
  \item\label{SizeCond}
  The size of the atom in the Morrey norm is bounded in terms of the length of its support, i.e.
  \begin{align*}
    \Morrey{a_{i,n}} \leq (C/\epsilon^2) \length(\gamma_{i,n})^{-1}
    .
  \end{align*}
  \item\label{CancellationCond}
  The atoms satisfy the cancellation condition
  \begin{align*}
    a_{i,n}(\omega)&=0  \text{ for all } \omega=d\pi \text{ for } \pi\in\mathcal{D}^1(E)
    .
  \end{align*}
\end{enumerate}
In addition, the atoms satisfy an additional normalization condition:
\begin{enumerate}[resume, ref=(\arabic*)]
  \item\label{NormalizationCond}
  The atoms $a_{i,n}$ are normalized so that $\mathbb{M}(a_{i,n})\leq 1$. 
\end{enumerate}
We remark that \ref{NormalizationCond} is not a size condition in the usual sense, as the inequality does not involve the size of the support.  
However, this condition ensures that if $T$ is any limit of the form
\begin{align}\label{decomposition_by_atoms}
  T = \lim_{n \to \infty} \sum_{i=1}^n \lambda_{i,n} a_{i,n}
  ,
\end{align}
where as in \Cref{approximationMetric} the convergence is pointwise on $\mathcal{D}^1(E)$, then $T$ defines an element $T \in \mathcal{M}_1(E)$.  
In particular, it follows that the atomic decomposition from \Cref{approximationMetric} characterizes the space $\set{T\in\mathcal{M}_1(E)\colon \partial T=0}$ of 1-currents without boundary.

\begin{remark}
  The convergence \eqref{decomposition_by_atoms} contrasts with a common idea of an atomic decomposition, to represent
  \begin{align}\label{classical_atomic_decomposition}
    T=\sum_{i=1}^\infty \lambda_i a_i.
  \end{align}
  It is not possible in general to write $T \in \mathcal{M}_1(E)$ in this way with $a_i$ supported on Lipschitz curves.  
  Indeed, if the $\{\lambda_i\}_{i\in \mathbb{N}}$ are a fixed absolutely summable sequence and $\{a_i\}_{i \in \mathbb{N}}$ satisfy $\mathbb{M}(a_{i})\leq 1$, then the series on the right-hand side of \eqref{classical_atomic_decomposition} converges in norm in the space of $1$-currents. 
  In particular, any such limit $T$ is supported on a one-dimensional set.  
\end{remark}  

\subsubsection{Interpreting the Morrey norm}
Measuring the size of atoms in terms of the Morrey norm is unfamiliar, but nonetheless natural.  
For a fixed curve $\gamma$, having a finite Morrey norm \eqref{BallGrowthConstant} means that the length of the portion of $\gamma$ in a ball grows no faster than a fixed constant times the radius of the ball.  
Intuitively, this quantifies the idea that $\gamma$ is genuinely one-dimensional:  for instance, the length measure of such a curve $\gamma$ -- the image of Lebesgue measure under the mapping $\gamma$ when parametrized by arc length -- is Ahlfors-David $1$-regular.  
The significance of the bound \eqref{morrey_bound} is to extend this  quantitative control uniformly over the curves $\{\gamma_{i,n}\}$ in \Cref{approximationMetric}.  
This uniformity is crucial in the proofs from \cites{HS,HRS} of Sobolev embeddings and estimates for elliptic PDE for the space of divergence free measures in the Euclidean setting.  
Our \Cref{approximationMetric} paves the way for the study of analogous PDE on a Riemannian manifold, which is the subject of a forthcoming work. 

Note that
\begin{align*}
  \operatorname*{diam} \gamma \leq  \length(\gamma) \leq \frac{1}{2} \Morrey{\gamma} \operatorname*{diam} \gamma
  .
\end{align*}
Thus the size condition \ref{SizeCond} implies a condition that more closely resembles that for classical atoms:
\begin{align*}\tag{$2'$}\label{SizePrime}
  \Morrey{a_{i,n}} \leq C' (\operatorname*{diam} \operatorname*{supp} a_{i,n})^{-1}
  .
\end{align*}
When the uniform Morrey bound \eqref{morrey_bound} is in force, this form of the bound is equivalent up to constants to the size condition \ref{SizeCond}.

\subsubsection{A comparison with classical atoms and dimension stability}

In the Euclidean setting, the support, size, and cancellation conditions can be expressed differently to give a more plain comparison with classical atoms.  
In particular, for any atom $a$ supported on a piecewise geodesic curve $\gamma$, there exists a ball $B$ containing $\operatorname*{supp} \gamma$ for which $\operatorname*{diam} B \approx \operatorname*{diam} \gamma$. 
Using $B$, we can formulate weaker analogues of the conditions \ref{SupportCond}, \ref{SizeCond}/\eqref{SizePrime}, \ref{CancellationCond} and \ref{NormalizationCond} without reference to $\gamma$:
\begin{align*}
  \tag{$1'$}\label{SupportPrime}
  \operatorname*{supp} a &\subset B
  ;
  \\
  \tag{$2''$}\label{SizeDoublePrime}
  \Morrey{a} &\leq \frac{C''}{\operatorname*{diam} B}
  ;
  \\
  \tag{$3'$}\label{CancellationPrime}
  \int_{B} \nabla \varphi \cdot da &= 0 
  \quad\text{for all $\varphi \in C^1_c(\mathbb{R}^n;\mathbb{R}^n)$}
  ;
  \\
  \tag{$4'$}\label{NormalizationPrime}
  \|a\|_{M_b} &\leq 1
  .
\end{align*}
Even without assuming that the atom $a$ is given by a curve $\gamma$, the weaker cancellation condition \eqref{CancellationPrime} still implies that $a$ is divergence free.  In particular, the seemingly weaker conditions on atoms \eqref{SupportPrime}, \eqref{SizeDoublePrime}, \eqref{CancellationPrime}, and \eqref{NormalizationPrime} are still sufficient to imply that any limit of the form \eqref{decomposition_by_atoms} defines a divergence free measure on $\mathbb{R}^n$ (and therefore continue to characterize the space of such measures).

The cancellation conditions \ref{CancellationCond}/\eqref{CancellationPrime} can be replaced by the still weaker condition 
\begin{align*}\tag{$3''$}\label{CancellationDoublePrime}
  \int_{B} da &=0
  .
\end{align*}
Unlike \ref{CancellationCond}/\eqref{CancellationPrime}, the condition \eqref{CancellationDoublePrime} no longer implies that $a$ is divergence free.
However, the combination of \eqref{SupportPrime}, \eqref{SizeDoublePrime}, \eqref{CancellationDoublePrime}, and \eqref{NormalizationPrime} is the most immediate for comparison with the Hardy space atoms, which satisfy the support, size, and cancellation conditions
\begin{align*}
  \operatorname*{supp} a &\subset B
  ;
  \\
  \|a\|_{L^\infty} &\leq \frac{C}{(\operatorname*{diam} B)^n}
  ;
  \\
  \int_{B} da &=0.
\end{align*}
\begin{remark}
  Note that for Hardy space atoms, only three conditions are needed as the condition \eqref{NormalizationPrime} can be deduced as a consequence of the stronger size bound imposed in that setting.  
\end{remark}
The conditions \eqref{SupportPrime}, \eqref{SizeDoublePrime}, \eqref{CancellationDoublePrime}, and \eqref{NormalizationPrime}, and their comparison with the Hardy space suggests, that it would be fruitful to study spaces in which the atoms satisfy a modified size condition 
\begin{equation*}\tag{$2''_\beta$}\label{SizePowerp}
  \|a\|_{\mathcal{M}^\beta} \leq \frac{C_\beta''}{(\diam B)^\beta}
\end{equation*}
combining a Morrey norm bound with a different power of the diameter.  
It turns out that a slighty weaker signed version of the condition \eqref{SizePowerp}, along with  \eqref{SupportPrime}, \eqref{CancellationDoublePrime}, and \eqref{NormalizationPrime}, are sufficient to ensure that any limit of the form \eqref{decomposition_by_atoms} has support on a set whose lower Hausdorff dimension is at least $\beta$.
The resulting function spaces are therefore called dimension stable spaces, see \cite{DS}*{Theorem H} and the surrounding discussion.

\subsubsection{Overview of the proof}
The proof of \Cref{approximationMetric} relies on two main components: a representation of integral currents without boundary in terms of closed curves; and an algorithm to convert a closed curve into a collection of closed curves satisfying uniform bounds on their Morrey norm.

Generalizing work of S.\ Smirnov~\cite{Smirnov} on divergence free measures in the Euclidean setting, Paolini and Stepanov \cite{PS} have shown that in a complete separable metric space, a 1-current without boundary can be represented as an integral with respect to a certain measure on the space of Lipschitz curves.
These curves under this measure are generally not closed, but they can be considered to be closed ``on average'' in the sense that the induced measures for the starting and ending points coincide.
In \Cref{CurrentsAndCurves} we argue that these curves can be represented by sequences of closed piecewise-geodesic curves: see \Cref{bbassertion}.
This is similar to an observation of Bourgain and Brezis \cites{BourgainBrezis2004,BourgainBrezis2007} for the Euclidean case, proved in \cite{GHS}.  

The crucial part of the argument is to show that a closed piecewise-geodesic curve can itself be represented by an equivalent collection of closed piecewise-geodesic curves satisfying uniform Morrey bounds.
This is our \Cref{surgery-lem}, which we call the Surgery Lemma, and \Cref{surgery-eta}.
This result does not depend on the connection to 1-currents, and for clarity of exposition we state and prove this result first, in \Cref{surgery}.

Our statement of the Surgery Lemma differs from the corresponding statement in \cite{HS}, on which our result is based.
In \cite{HS}*{Lemma~5.1}, a smooth closed curve in Euclidean space is represented by a collection of piecewise-smooth closed curves, and the total length increases by a large universal constant.
By contrast, \Cref{surgery-lem} produces closed piecewise-geodesic curves in a more general metric space.
The proof of \Cref{surgery-lem} has been simplified considerably compared to the original proof of \cite{HS}*{Lemma~5.1}.
We introduce parameters $\epsilon$ and $n$, which were implicitly set as arbitrary constants in \cite{HS}*{Lemma~5.1}, and as a result \Cref{approximationMetric} allows the total length of the curves to be within a factor $1+\epsilon$ of the mass of the initial $1$-current, at the cost of a factor $C/\epsilon^2$ in the Morrey norm of the atoms. 

In Euclidean space, the piecewise-geodesic curves are precisely the polygonal curves.
Thus \Cref{approximationMetric} applied in Euclidean space yields \cite{HS}*{Theorem~1.5}, with more precise control over the dependence on constants.

\Cref{preliminaries} introduces notation for curves and geodesics.
\Cref{surgery} states and proves the Surgery Lemma~\ref{surgery-lem}, and its consequence \Cref{surgery-eta} that more closely matches the structure of \Cref{approximationMetric}.
\Cref{CurrentsAndCurves} gives further details on differential forms and metric currents and proves \Cref{bbassertion}.
In \Cref{ProofOfMainResult} these elements are combined to complete the proof of \Cref{approximationMetric}.
In \Cref{appendix} we prove a separability theorem for the space of Lipschitz functions on a complete separable metric space as well as several continuity results for metric $1$-forms, results which may be of independent interest, though are specifically useful for our purposes in \Cref{CurrentsAndCurves} in the proof of \Cref{bbassertion}.  

\section{Preliminaries}\label{preliminaries}

\subsection{Lipschitz curves}\label{LipschitzCurves}
By a Lipschitz curve in $E$ we mean a function $\gamma\colon[a,b]\to E$, where $a\leq b$, such that $\sup\set{d(\gamma(t),\gamma(t'))/ |t-t'|}<\infty$, the supremum being taken over distinct $t,t'\in[a,b], t\neq t'$.
The length of the curve is 
\begin{align*}
  \length(\gamma) := \sup\set{\sum_{i=1}^k d(\gamma(s_{i-1}), \gamma(s_i)) \colon a=s_0<s_1<\dotsb<s_{k-1}<s_k=b} 
  ,
\end{align*}
and it is easily seen that a Lipschitz curve has finite length.
Given a subinterval $[t,t']\subset[a,b]$, $t<t'$, we write $\gamma|_{[t,t']}$ for the Lipschitz curve formed by restricting to the subinterval $[t,t']$.

Define $\tilde{\Theta}(E)$ to be the set of equivalence classes of Lipschitz-continuous curves in $E$ modulo increasing reparametrization, where two curves $\gamma_i\colon[a_i,b_i]\to E$ are equivalent when there is a bijective increasing function $\phi\colon [a_1,b_1]\to [a_2,b_2]$ such that $\gamma_1 = \gamma_2\circ\phi$.
It is easily seen that the length of a Lipschitz curve is invariant under reparametrization.
Thus any equivalence class in $\tilde{\Theta}(E)$ contains a unique representative $\gamma$ defined on the interval $[0,\length(\gamma)]$ and parametrized by arc length in the sense that $\length(\gamma|_{[0,t]})=t$ for all $t\in [0,\length(\gamma)]$.
We may therefore identify $\tilde{\Theta}(E)$ with 
\begin{equation*}
  \Theta(E) = \set{\gamma\colon [0,\length(\gamma)]\to E\text{ such that $\gamma$ is parametrized by arc length}}
\end{equation*}
and for definiteness, we will deal henceforth primarily with $\Theta(E)$.
For instance, we will interpret $\gamma|_{[t,t']}$ modulo reparametrization as the curve $\tilde{\gamma}\in\Theta(E)$ defined by $\tilde{\gamma}\colon [0,t'-t]\to E$, $\tilde{\gamma}(s)=\gamma(t+s)$. 

Let $\gamma_1,\dotsc,\gamma_k\in\Theta(E)$ be Lipschitz curves such that the endpoint of $\gamma_i$ is the initial point of $\gamma_{i+1}$ for all $i=1,\dotsc,k-1$.
Then we can concatenate the curves $\gamma_1,\dotsc,\gamma_k$ to form $\gamma\in\Theta(E)$ that is also a Lipschitz curve.
In detail, if we set $s_i=\sum_{j=1}^i \length(\gamma_j)$ for $i=0,\dotsc,k$, then $\gamma(t)=\gamma_i(t-s_{i-1})$ for all $t\in[s_{i-1},s_i]$.
Note that $\gamma(t)$ is well-defined at $t\in\set{s_1,\dotsc,s_{k-1}}$ due to the assumption on the endpoints of the curves $\gamma_1,\dotsc,\gamma_k$; we will impose this condition on the endpoints as a standing requirement whenever we discuss concatenations.

Given $\gamma\in\Theta(E)$ parametrized by arc length, we define $\mu_{\gamma}$ to be the image of Lebesgue measure under the mapping $\gamma$, i.e., the Borel measure defined by
\begin{align}
  \label{mugammaFormula}
  \mu_{\gamma}(B): = |  \{ t\in[0,\length(\gamma)] \colon \gamma(t) \in B  \} |
  ,
\end{align}
where $|\cdot|$ denotes the Lebesgue measure.
The Morrey norm also applies to Borel measures $\mu$ on $E$, with
\begin{equation}\label{MorreyForMeasures}
  \Morrey{\mu} = \sup_{r>0, x\in E} \frac{\mu(B_r(x))}{r}
  ,
\end{equation}
and the definition \eqref{BallGrowthConstant} of $\Morrey{\gamma}$ can equivalently be expressed as $\Morrey{\gamma}=\Morrey{\mu_\gamma}$.

\begin{lemma}\label{MorreyIsNorm}
  If $\gamma\in\Theta(E)$ is the concatenation of curves $\gamma_1,\dotsc,\gamma_k\in\Theta(E)$, then 
  $[[\gamma]] = [[\gamma_1]] + \dotsb + [[\gamma_k]]$ and   $\Morrey{\gamma} \leq \Morrey{\gamma_1} + \dotsb + \Morrey{\gamma_k}$.
\end{lemma}

\begin{proof}
It is easily seen that $\mu_\gamma = \mu_{\gamma_1} + \dotsb + \mu_{\gamma_k}$ because the length in the definition \eqref{mugammaFormula} of $\mu_\gamma$ can be split across the $k$ subintervals (disjoint except at their endpoints) corresponding to the concatenated curves $\gamma_1,\dotsc,\gamma_k$.
Since $\Morrey{\gamma}=\Morrey{\mu_\gamma}$ by definition, the inequality follows from the fact that $\Morrey{\cdot}$ is a norm on measures.

We will defer the proof that $[[\gamma]] = [[\gamma_1]] + \dotsb + [[\gamma_k]]$ until we have formally defined metric currents in \Cref{differentialforms}, but we note that it follows immediately from the definition \eqref{currentcurve} of $[[\gamma]]$.
\end{proof}

\subsection{Geodesic curves}

For the purposes of this paper, we say that a curve $\gamma\in\Theta(E)$ parametrized by arc length is a \emph{geodesic} if $d(\gamma(s),\gamma(t))=|s-t|$ for all $s,t\in[0,\length(\gamma)]$, i.e., if $\gamma$ is an isometry between $[0,\length(\gamma)]$ and its image in $E$.

We note that this definition of geodesic differs slightly from how the term is used in other contexts, where $\gamma$ may only be required to be a local isometry.
Thus, what we have called a geodesic means a ``length-minimizing geodesic.''

\begin{lemma}\label{ball-growth-geodesic}
  Let $\gamma\in\Theta(E)$ be a geodesic.
  Then $\gamma$ satisfies the ball-growth condition
  \[
  \Morrey{\gamma} = 2
  .
  \]
\end{lemma}
\begin{proof}
  A geodesic $\gamma$ gives an isometry from $[0,\length(\gamma)]$ to $E$.
  Thus if $t,t'\in[0,\length(\gamma)]$ have $\gamma(t),\gamma(t')\in B_r(x)$ for some $x\in E$, $r>0$, then $|t-t'|=d(\gamma(t),\gamma(t'))\leq 2r$.
  In particular, the set $\set{t\in[0,\length(\gamma)]\colon \gamma(t)\in B_r(x)}$ has diameter at most $2r$, so that $\mu_\gamma(B_r(x)) = |\set{t\in[0,\length(\gamma)]\colon \gamma(t)\in B_r(x)}|\leq 2r$.
  Equality is attained if we take $r=\tfrac{1}{2}\length(\gamma)$ and $x=\gamma(r)$: then $\gamma(t)\in B_r(x)$ for all $t\in[0,\length(\gamma)]$, and this interval has length $2r$ by construction.
  Therefore $\Morrey{\gamma}=\sup_{r>0,x\in E} \mu_\gamma(B_r(x))/r = 2$.
\end{proof}

Our assumption on the space $E$ means that for each unordered pair $\set{x,y}$ of points in $E$, there exists a (length-minimizing) geodesic $\gamma$ with endpoints $x,y$.
By considering both $\gamma$ and its reversal, we may therefore construct a pair of geodesics $\gamma_1,\gamma_2$ both of length $d(x,y)$, from $x$ to $y$ and from $y$ to $x$ respectively, and such that $[[\gamma_1]]+[[\gamma_2]]=0$.

For notational convenience, we will write this pair of curves as a function of the unordered pair $\set{x,y}$.
That is, $G_{x,y}$ will denote a geodesic of length $\length(G_{x,y})=d(x,y)$ with endpoints $G_{x,y}(0)=x$, $G_{x,y}(d(x,y))=y$, such that $G_{y,x}$ is the reversal of $G_{x,y}$ and
\begin{align}\label{ReversedGeodesics}
  [[G_{x,y}]] + [[G_{y,x}]] = 0
  .
\end{align}
We remark that in general, choosing a single such geodesic simultaneously for all choices of $\set{x,y}$ may require an appeal to the axiom of choice.

\subsection{Piecewise geodesic curves}

A piecewise geodesic means a concatenation of finitely many geodesic curves (with the requirement of compatible endpoints, as mentioned in \cref{LipschitzCurves}).
Equivalently, $\gamma\in\Theta(E)$ is a piecewise geodesic curve if there exists a partition $s_0=0<s_1<\dotsb<s_{k-1}<s_k=\length(\gamma)$ such that $\gamma|_{[s_{i-1},s_i]}$ is a geodesic for each $i=1,\dotsc,k$.
We will call such a partition $(s_0,\dotsc,s_k)$ a \emph{resolution of $\gamma$ into geodesic edges}.
We remark that resolutions into geodesic edges are not unique: for instance, refining such a partition yields another resolution into geodesic edges.
Once we have fixed a resolution into geodesic edges, we will call each subcurve $\gamma|_{[s_{i-1},s_i]}$ a \emph{geodesic edge} of $\gamma$.

Write $\Xgeo\subset \Theta(E)$ for the set of piecewise geodesic curves.

\begin{corollary}
  \label{C1-small-ball}
  Suppose $\gamma \in \Xgeo$ can be written as the concatenation of at most $k$ geodesic curves.
  Then $\mu_\gamma$ satisfies the ball growth condition
  \[
    \Morrey{\gamma} \leq 2k
    ,\qquad\text{i.e.,}\qquad
    \mu_\gamma(B_r(x)) \leq 2 k r\quad\text{for all $r>0,x\in E$}
    .
  \]
\end{corollary}
\Cref{C1-small-ball} follows immediately from \Cref{ball-growth-geodesic,MorreyIsNorm}.

\subsection{Closed curves and the basic cut operation}\label{BasicCut}

We call $\gamma\in\Theta(E)$ a closed curve if $\gamma(\length(\gamma))=\gamma(0)$, and write $\Theta_c(E)\subset\Theta(E)$ for the set of closed curves.
Then the definition of $\gamma|_{[t,t']}$, given in \Cref{LipschitzCurves} for the case $t<t'$, can be naturally extended to the case $t'<t$ while still respecting the orientation of $\gamma$.
Namely, for $\gamma\in\Theta_c(E)$ and $0\leq t'<t\leq\length(\gamma)$, we interpret $\gamma|_{[t,t']}$ as the concatenation of $\gamma|_{[t,\length(\gamma)]}$ and $\gamma|_{[0,t']}$, i.e., the curve $\tilde{\gamma}\in\Theta(E)$ defined by 
\begin{equation*}
  \tilde{\gamma}\colon [0,\length(\gamma)-t+t']\to E
  , 
  \quad 
  \tilde{\gamma}(s)=\begin{cases}
    \gamma(t+s), & 0\leq s<\length(\gamma)-t,\\
    \gamma(s-t), & \length(\gamma)-t \leq s \leq \length(\gamma)-t+t'
    .
  \end{cases}
\end{equation*}
In effect, we interpret $[0,\length(\gamma)]$ as the circle of length $\length(\gamma)$ traversed counter-clockwise with the endpoints $0$ and $\length(\gamma)$ identified, and we interpret $[t,t']$ as the counter-clockwise interval from $t$ to $t'$, either $[t,t']$ or $[t',\length(\gamma)]\cup[0,t]$ depending on whether $t\leq t'$ or $t>t'$.
We will call $[t,t']$ a sub-interval of $[0,\length(\gamma))]$ in both cases.
In a similar spirit we define
\begin{equation*}
  d_\gamma(t,t') = \min(\length(\gamma|_{[t,t']}),\length(\gamma|_{[t',t]}))
  ,
\end{equation*}
which we can interpret as the distance between $t$ and $t'$ in the circle.
We remark also that $\length(\gamma|_{[t,t']})+\length(\gamma|_{[t',t]})=\length(\gamma)$ for all distinct $t,t'\in[0,\length(\gamma)]$.

For consistency, we will also define $d_\gamma(s,s')=\length(\gamma|_{[\min(s,s'),\max(s,s')]})$ for a non-closed curve $\gamma\in\Theta(E)\setminus\Theta_c(E)$.

In \Cref{surgery} we will perform ``surgery'' on a closed curve, repeatedly cutting it into well-behaved pieces.
To this end, we first introduce a basic cut operation $C(\gamma,t,t')$, which splits a closed curve $\gamma\in\Theta_c(E)$ at two specified points $t,t'\in[0,\length(\gamma)]$ and returns a pair of closed curves formed from the two parts, each closed by a single geodesic edge traversed in opposite directions.
Specifically, $C(\gamma,t,t')$ is defined to be the pair $(\gamma',g)$ of closed curves, where $\gamma'$ is the concatenation of $\gamma|_{[t',t]}$ and $G_{\gamma(t),\gamma(t')}$, while $g$ is the concatenation of $\gamma|_{[t,t']}$ and $G_{\gamma(t'),\gamma(t)}$.
In particular,
\begin{gather*}
  [[\gamma]] = [[\gamma']] + [[g]]
  .
\end{gather*}
We write $\Xgeo_c = \Xgeo\cap\Theta_c(E)$ for the set of closed piecewise-geodesic curves.

\section{The Surgery Lemma} \label{surgery}

A key ingredient in the proof of \autoref{approximationMetric} is a metric analogue of the surgery lemma in \cite{HS}*{Lemma~5.1}. 
The Surgery Lemma shows that $\gamma\in\Xgeo_c$ can be represented by a sum of curves that are themselves closed and have uniformly bounded Morrey norms, with quantitative bounds on the amount of additional length required.

\begin{lemma}[Surgery]\label{surgery-lem}
  Let $E$ be a complete, separable, geodesic metric space, and let $0<\epsilon<1$ and $n\in\mathbb{N}$.
  Suppose $\gamma \in \Xgeo_c$ is a closed piecewise-geodesic curve.
  Then there exist $N\in\mathbb{N}$ and closed piecewise-geodesic curves $(\gamma_{j})_{j=1}^N$ in $\Xgeo_c$ such that
  \begin{enumerate}[ref=(\arabic*)]
    \item\label{item:SurgeryMeasureIdentity}
    the decomposition
    \begin{align*}
      [[\gamma]]= \sum_{j=1}^{N} [[\gamma_{j}]]
    \end{align*}
    holds;

    \item\label{item:SurgeryTotalLength}
    the total length of the curves in the decomposition satisfies
    \begin{align*}
      \sum_{j=1}^N \length(\gamma_{j}) \leq \left( 1 + \frac{2\epsilon}{1-\epsilon} + \frac{12 \epsilon^{-1}}{n(1-\epsilon)} \right) \length(\gamma) ;
    \end{align*}
    
    \item\label{item:SurgeryBallGrowth}
    each measure $\mu_{\gamma_j}$ satisfies the ball growth condition
    \begin{align*}
      \Morrey{\gamma_j}
      \leq 4\epsilon^{-1}+2n+10
      .
    \end{align*}
  \end{enumerate}
\end{lemma}

The curves $(\gamma_{j})_{j=1}^N$ from \Cref{surgery-lem} will be determined from $\gamma$ via an algorithm controlled by parameters $\epsilon$ and $n$.
By choosing these parameters appropriately, the curves $(\gamma_{j})_{j=1}^N$ can be chosen to have a small relative increase in total length, as in the following result.

\begin{corollary}\label{surgery-eta}
  There exists a universal constant $C'$ such that if $E$ is a complete, separable, geodesic metric space, $0<\eta<1$, and $\gamma \in \Xgeo_c$ is a closed piecewise-geodesic curve, then there exist $N\in\mathbb{N}$ and closed piecewise-geodesic curves $(\gamma_{j})_{j=1}^N$ in $\Xgeo_c$ such that
  \begin{enumerate}[ref=(\arabic*)]
    \item\label{item:SurgeryMeasureIdentity-eta}
    the decomposition
    \begin{align*}
      [[\gamma]]= \sum_{j=1}^{N} [[\gamma_{j}]]
    \end{align*}
    holds;
    
    \item\label{item:SurgeryTotalLength-eta}
    the total length of the curves in the decomposition satisfies
    \begin{align*}
      \sum_{j=1}^N \length(\gamma_{j}) \leq \left( 1 + \eta \right) \length(\gamma) ;
    \end{align*}
    
    \item\label{item:SurgeryBallGrowth-eta}
    each measure $\mu_{\gamma_j}$ satisfies the ball growth condition
    \begin{align*}
      \Morrey{\gamma_j}
      \leq \frac{C'}{\eta^2}
      .
    \end{align*}
  \end{enumerate}
\end{corollary}
\Cref{surgery-eta} follows immediately from \Cref{surgery-lem} with $\epsilon=c\eta$ and $n=\lceil\epsilon^{-2}\rceil$ for a sufficiently small absolute constant $c>0$.

\subsection{Regularity conditions on piecewise-geodesic curves}

\Cref{ball-growth-geodesic} shows that each geodesic edge in a piecewise-geodesic curve $\gamma\in\Xgeo$ has Morrey norm 2.
If too many such geodesic edges are close together (for instance, if $\gamma$ consists of the same geodesic traversed repeatedly in opposite directions) then the Morrey norm of $\gamma$ may become large, as in the upper bound from \Cref{C1-small-ball}.
However, if the geodesic edges remain substantially separated across large spatial scales, the Morrey norm of $\gamma$ may stay considerably smaller.

This intuition is formalized in \Cref{LSI-def,delta-def}.
In both definitions, $\delta$ will represent the spatial scale of small geodesic edges from $\gamma$, and all our estimates will be uniform in $\delta>0$, whereas $\epsilon$ and $n$ will be the fixed parameters from \Cref{surgery-lem}.

\begin{definition}[Large-scale invertibility]\label{LSI-def}
  Let $\delta>0$ and $0<\epsilon<1$.
  We say that a piecewise-geodesic curve $\gamma\in\Xgeo$ is \emph{$(\delta,\epsilon)$-large-scale-invertible}, abbreviated \emph{$(\delta,\epsilon)$-l.s.i.}, if
  \begin{equation}
    \label{bilip-2}
    d(\gamma(s),\gamma(t)) > \epsilon d_\gamma(s,t)
    \quad \text{ for all } s,t \text{ with }d_\gamma(s,t)\geq \delta
    .
  \end{equation}
We write $\Xlsi(\delta,\epsilon)$ for the set of all $(\delta,\epsilon)$-l.s.i.\ curves.
\end{definition}

\begin{lemma}
    \label{large-ball-lem}
   If $\gamma\in\Xlsi(\delta,\epsilon)$, 
   then the measure $\mu_\gamma$ satisfies the ball growth condition
    \[
        \sup_{r\geq\delta/2, x\in E}\frac{\mu_\gamma(B_r(x))}{r} \leq 4 \left( 1+\epsilon^{-1} \right)
        .
    \]
\end{lemma}

\begin{proof}
Let $x\in E$ and $r\geq \delta/2$.  
If $\mu_\gamma(B_r(x)) > 0$, then there exists $s\in [0,\length(\gamma)]$ such that $\gamma(s)\in B_r(x)$.
It follows that $B_r(x)\subset B_{2r}(\gamma(s))$, and so
\[
  \mu_\gamma(B_r(x)) \leq \mu_\gamma(B_{2r}(\gamma(s)))
  = |\{t\in[0,\length(\gamma)] \colon d(\gamma(s),\gamma(t)) \leq 2r\}|
  .
\]
As $\gamma$ satisfies~\eqref{bilip-2}, we have the estimate
\begin{align*}
  &|\{t\in[0,\length(\gamma)] \colon d(\gamma(s),\gamma(t)) \leq 2r\}| 
  \\&\quad
  = |\{t\in[0,\length(\gamma)] \colon d(\gamma(s),\gamma(t)) \leq 2r, d_\gamma(s,t) <\delta\}| \\
  &\qquad +  |\{t\in[0,\length(\gamma)] \colon d(\gamma(s),\gamma(t)) \leq 2r, d_\gamma(s,t) \geq \delta \}|\\
  &\quad\leq 2 \delta + |\{ t\in[0,\length(\gamma)] \colon d_\gamma(s,t) < 2\epsilon^{-1} r\}|\\
  &\quad\leq 4r+ 4\epsilon^{-1}r
  .
  \qedhere
\end{align*}

\end{proof}

Note that if $\delta$ could be replaced by 0, the conclusion of \Cref{large-ball-lem} would become
\begin{align*}
  \Morrey{\gamma}  \leq 4 \left( 1+\epsilon^{-1} \right)
  ,
\end{align*}
similar to the desired bound in \Cref{surgery-lem}\ref{item:SurgeryBallGrowth}.  
In general, to account for small radii, we introduce the following class of curves, which are constrained to have a bounded number of small geodesic edges in moderately small parts of the curve.
\begin{definition}[$(\delta,\epsilon,n)$-curve]
  \label{delta-def}
  Let $\gamma\in\Xgeo$ be a piecewise-geodesic curve.
  We say that $\gamma$ is a $(\delta,\epsilon,n)$-curve if it has a resolution $(s_0,\dotsc,s_k)$ into geodesic edges such that, for all subintervals $[t,t']\subset [0,\length(\gamma)]$ of length at most $2\epsilon^{-1}\delta$, 
  \begin{align*}
    \# \set{ j\in\set{1,\dotsc,k} \colon [s_{j-1},s_j]\subset[t,t']\text{ and }\length(\gamma|_{[s_{j-1},s_j]}) < \delta } \leq n
    .
  \end{align*}
  We denote by $\Xden(\delta,\epsilon,n)$ the set of all $(\delta,\epsilon,n)$-curves.
\end{definition}

In words, \Cref{delta-def} says that in a $(\delta,\epsilon,n)$-curve, every subinterval of length at most $2\epsilon^{-1}\delta$ can be written as a concatenation of geodesic edges, at most $n$ of which have length less than $\delta$, plus up to two partial geodesic edges overlapping the endpoints of the subinterval.
We remark that we obtain an equivalent definition if we also require in \Cref{delta-def} that $t,t'\in\set{s_0,\dotsc,s_k}$, i.e., that  $\gamma|_{[t,t']}$ should be a concatenation of geodesic edges.

It is also useful to count the number of edges with length smaller than $\delta$.
\begin{definition}\label{nsmallpieces-def}
  For a piecewise-geodesic curve $\gamma\in\Xgeo$, define
  \begin{align*}
    \nsmallpieces(\gamma,\delta):=  \min_{(s_0,\dotsc,s_k)} \# \set{ j\in\set{1,\dotsc,k} \colon  \length(\gamma|_{[s_{j-1},s_j]}) < \delta }
    ,
  \end{align*}
  where the minimum is over resolutions $(s_0,\dotsc,s_k)$ of $\gamma$ into geodesic edges.
\end{definition}
Note that if $\gamma$ fails to be a $(\delta,\epsilon,n)$-curve then necessarily $\nsmallpieces(\gamma,\delta) \geq n+1$, but the converse need not hold.

\begin{lemma}
    \label{full-ball-lem}
   If $0<\epsilon<1$, $n\in\mathbb{N}$, and $\gamma\in\Xden(\delta,\epsilon,n)\cap\Xlsi(\delta,\epsilon)$, 
   then $\gamma$ satisfies the ball growth condition
    \[
        \Morrey{\gamma} \leq 4\epsilon^{-1}+2n+4
        .
    \]
\end{lemma}

\begin{proof}
By the assumption $\gamma\in\Xlsi(\delta,\epsilon)$, \Cref{large-ball-lem} applies and shows that
\begin{align*}
 \sup_{r \geq \delta/2, x\in E}\frac{\mu_\gamma(B_r(x))}{r} \leq 4 \left( 1+\epsilon^{-1} \right) r
 .
\end{align*}
Now consider $r<\delta/2$.
As in the proof of \Cref{large-ball-lem}, if $\mu_\gamma(B_r(x))>0$, we find $s \in[0,\length(\gamma)]$ such that $\gamma(s) \in B_r(x)$.
We next define the interval
\begin{align*}
  I:= \{t \in[0,\length(\gamma)]: d_\gamma(s,t) <  \epsilon^{-1}\delta\}.
\end{align*}
We claim that if $\gamma(t) \in B_{2r}(\gamma(s))$ then $t \in I$.
To see this, split according to whether $d_\gamma(s,t)<\delta$ or $d_\gamma(s,t)\geq \delta$.
If $d_\gamma(s,t)<\delta$, then the assumption $0<\epsilon<1$ gives $d_\gamma(s,t)<\delta<\epsilon^{-1}\delta$ and $t\in I$.
On the other hand, if $d_\gamma(s,t)\geq\delta$, then \eqref{bilip-2} applies and yields
\begin{equation*}
  d_\gamma(s,t) < \epsilon^{-1} d(\gamma(s),\gamma(t)) \leq \epsilon^{-1} (2r)<\epsilon^{-1}\delta
\end{equation*}
by the assumption $r<\delta/2$.

In particular, the part of $\gamma$ lying in $B_{2r}(\gamma(s))$ is a subset of $\gamma\vert_I$.
Now fix a resolution of $\gamma$ into geodesic edges.
The length of $\gamma|_I$ is at most $2\epsilon^{-1}\delta$, and therefore $\gamma\vert_I$ fully contains at most $2\epsilon^{-1}$ geodesic edges of length at least $\delta$.
On the other hand, $I$ has length at most $2\epsilon^{-1}\delta$.
Since $\gamma$ is a $(\delta,\epsilon,n)$-curve, $\gamma|_{I}$ contains at most $n$ geodesic edges of $\gamma$ of length less than $\delta$.  
Finally $\gamma|_I$ may contain parts of up to 2 other edges, up to one edge containing each of the endpoints of $\gamma|_I$.

In total, $\gamma|_I$ is contained in the union of $k\leq 2\epsilon^{-1}+n+2$ geodesic edges.
By \Cref{C1-small-ball}, 
\begin{align*}
  \mu_\gamma(B_r(x)) \leq \mu_\gamma(B_{2r}(\gamma(s))) \leq   (2\epsilon^{-1}+n+2)(2r)
  \leq (4\epsilon^{-1}+2n+4)r.
\end{align*}
Together with the estimate for $r\leq \delta/2$ from \Cref{large-ball-lem}, this yields the claim.
\end{proof}

\subsection{Type I cut operation}\label{ss-TypeI}
By \Cref{full-ball-lem}, a $(\delta,\epsilon,n)$-curve that is $(\delta,\epsilon)$-large-scale-invertible admits a ball growth condition.
Given a general curve, the idea is to perform ``surgery'' on the curve, using the basic cut operation $C(\gamma,t,t')$ from \Cref{BasicCut} in two distinct ways.  

In a Type I cut operation, the input is a $(\delta,\epsilon,n)$-curve that is not $(\delta,\epsilon)$-l.s.i., which is split into one slightly shorter curve and one well-behaved curve satisfying a ball growth condition.

\begin{lemma}[Type I cut operation]
  \label{cut_I}
  Let $\epsilon \in (0,1)$ and $n\in\mathbb{N}$.  
  There exist 
  \begin{align*}
    C_{\mathrm{I}}(\gamma) &\colon \Xden(\delta,\epsilon,n) \setminus \Xlsi(\delta,\epsilon)
    \to \Xgeo \times \Xgeo 
    \\
    \text{and}\quad \beta(\gamma) &\colon \Xden(\delta,\epsilon,n) \setminus \Xlsi(\delta,\epsilon) \to [\delta, \tfrac{1}{2}\length(\gamma)]
  \end{align*}
  with the following properties.
  Abbreviate $(\gamma',g):=C_{\mathrm{I}}(\gamma)$.
  Then
  \begin{align}\label{SameSumTypeI}
    [[\gamma]] = [[\gamma']] + [[g]]
    ,
  \end{align}
  the curve $g$ admits the ball growth bound
  \begin{align}\label{MorreyTypeI}
    \Morrey{g} \leq 4\epsilon^{-1}+2n+10
    ,
  \end{align}
  the curves $\gamma',g$ satisfy the length bounds
  \begin{align}
    \length(\gamma') &\leq \length(\gamma) -(1-\epsilon)\beta(\gamma)
    \label{LengthBoundTypeIgamma}
    , 
    \\
    \length(\gamma') +\length(g) &\leq \length(\gamma) + 2\epsilon \beta(\gamma)
    \label{LengthBoundTypeItotal}
    ,
  \end{align}
  and the number of small edges in the curve $\gamma'$ has the bound
  \begin{align}\label{SmallCountBoundTypeI}
    \nsmallpieces(\gamma',\delta) \leq \nsmallpieces(\gamma,\delta) +3
    .
  \end{align}
\end{lemma}

As we shall explain in the proof of \Cref{surgery-lem}, we will use $\gamma'$ as the input at the next stage of an iterative procedure.
Thus the bounds \eqref{LengthBoundTypeIgamma} and $\beta(\gamma)\geq\delta$, which show that a Type I cut operation reduces the length by an amount bounded away from 0, will help to guarantee that the iterative procedure must terminate.
Meanwhile the curve $g$, which already satisfies a ball growth condition, will become one of the output curves in \Cref{surgery-lem}.

\begin{proof}
Consider a closed curve $\gamma\in\Xden(\delta,\epsilon,n)\setminus \Xlsi(\delta,\epsilon)$ that is a $(\delta,\epsilon,n)$-curve but is not $(\delta,\epsilon)$-large-scale-invertible.
Note that such a curve must have $\length(\gamma) \geq 2\delta$.

The failure of $\gamma$ to be $(\delta,\epsilon)$-large-scale-invertible
implies that there exist $s,s'\in[0,\length(\gamma)]$
with $d_\gamma(s,s') \geq \delta$ and $d(\gamma(s),\gamma(s')) \leq \epsilon d_{\gamma}(s,s')$.
We may therefore define
\begin{equation}\label{betagammainf}
  \beta(\gamma) = \inf \{ \length(\gamma|_{[s,s']}) \colon s,s' \in [0,\length(\gamma)], d_{\gamma}(s,s') \geq \delta
  \text{ and } d(\gamma(s),\gamma(s')) \leq \epsilon d_{\gamma}(s,s')\}
  ,
\end{equation}
where the infimum is over a non-empty set.

Since $\length(\gamma|_{[s,s']})\geq d_\gamma(s,s')\geq\delta$, we have $\beta(\gamma)\geq\delta$ by construction.
The constraints in the infimum \eqref{betagammainf} are unaffected by interchanging $s$ and $s'$, so we may assume without loss of generality that $\length(\gamma|_{[s,s']})\leq \length(\gamma|_{[s',s]})$.
In this case $\length(\gamma|_{[s,s']})+\length(\gamma|_{[s',s]}) = \length(\gamma)$ implies that $\length(\gamma|_{[s,s']})\leq \tfrac{1}{2}\length(\gamma)$, and we conclude that $\beta(\gamma)\leq\tfrac{1}{2}\length(\gamma)$.

Since $\length(\gamma|_{[s,s']})$, $d_\gamma(s,s')$ and $d(\gamma(s),\gamma(s'))$ are continuous functions of $s,s'$, the set of pairs $(s,s')$ that appear in \eqref{betagammainf} is a non-empty compact subset of $[0,\length(\gamma)]^2$.
Hence the infimum is attained and we may choose in some definite way a pair $(t,t')=(t(\gamma), t'(\gamma))$ attaining the infimum; for instance, we may choose $(t,t')$ to be the lexicographically-smallest pair attaining the infimum.
Then $\beta(\gamma)=d_\gamma(t(\gamma),t'(\gamma))$, and we set $C_{\mathrm{I}}(\gamma) = C(\gamma, t(\gamma), t'(\gamma)) := (\gamma',g)$.
Thus $g$ consists of a curve $\gamma|_{[t,t']}$ of length $\beta(\gamma)$, closed by a single geodesic edge; and $\gamma'$ consists of the rest of $\gamma$, with $\gamma|_{[t,t']}$ replaced by a single geodesic edge.
The identity \eqref{SameSumTypeI} follows as a property of the basic cut operation $C(\gamma,t,t')$.

For the length bounds \eqref{LengthBoundTypeIgamma}--\eqref{LengthBoundTypeItotal}, we note that $\gamma'$ is formed from $\gamma$ by removing a part $\gamma|_{[t,t']}$ of length $\beta(\gamma)$ and inserting a geodesic edge of length $d(\gamma(t),\gamma(t'))$.
Since $(t,t')$ attain the infimum in \eqref{betagammainf} we have the bound $d(\gamma(t),\gamma(t'))\leq\epsilon d_\gamma(t,t')$, so that $\length(\gamma')-(\length(\gamma)-\beta(\gamma))=d(\gamma(t),\gamma(t'))\leq \epsilon d_\gamma(t,t')$, as in \eqref{LengthBoundTypeIgamma}.
Similarly, $\gamma'$ and $g$ collectively include the entire path $\gamma$ along with two geodesic edges of length $d(\gamma(t),\gamma(t'))$, so the bound \eqref{LengthBoundTypeItotal} follows by the same argument.

For the bound \eqref{SmallCountBoundTypeI} on small edges, note that $\gamma'$ can be formed using at most 3 edges that are not edges of $\gamma$: up to two partial edges between $\gamma(t),\gamma(t')$ and the respective closest endpoints of edges contained in $\gamma|_{[t',t]}$; and the geodesic edge $[\gamma(t'),\gamma(t)]$.
Each of those 3 edges may have length less than $\delta$, and $m(\gamma',\delta)\leq m(\gamma,\delta)+3$ follows.

It remains to prove the ball growth bound \eqref{MorreyTypeI} for $g$.
Fix a resolution of $\gamma$ into geodesic edges satisfying the properties in \Cref{delta-def}.

We first handle the case where $\beta(\gamma)=\delta$, i.e., where the curve $\gamma|_{[t,t']}$ has length $\delta$.
Since $\epsilon<1$, the curve $\gamma|_{[t,t']}$ has length less than $2\epsilon^{-1}\delta$, so we may apply the definition of $(\delta,\epsilon,n)$-curve to conclude that $\gamma|_{[t,t']}$ contains at most $n$ edges of $\gamma$ with lengths smaller than $\delta$.
In addition, $\gamma|_{[t,t']}$ may contain parts of up to 2 other edges of $\gamma$; or $\gamma|_{[t,t']}$ may instead consist of a single edge of length $\delta$.
Either way, $\gamma|_{[t,t']}$ can be formed as a concatenation of $k\leq n+2$ geodesics, so that \Cref{C1-small-ball} gives $\Morrey{g}\leq 2(n+2)$, which implies \eqref{MorreyTypeI}.

Otherwise, we may assume $\beta(\gamma)>\delta$.
By minimality, every strictly smaller subinterval $[s,s']\subsetneq [t,t']$ of length at least $\delta$ must satisfy $d(\gamma(s),\gamma(s')) > \epsilon d_\gamma(s,s')$.
Since $\beta(\gamma)>\delta$, we may find a sequence of subintervals increasing towards $[t,t']$ and all having length at least $\delta$.
By continuity it follows that $d(\gamma(t),\gamma(t'))\geq\epsilon d_\gamma(t,t')$, so that 
\begin{equation}\label{gammatt'lsi}
    d(\gamma(s),\gamma(s'))\geq\epsilon d_\gamma(s,s') \quad\text{for all $s,s'\in[t,t']$ such that }d_\gamma(s,s')\geq \delta
  .
\end{equation}
Since $\length(\gamma|_{[t,t']})=\beta(\gamma)\leq\tfrac{1}{2}\length(\gamma)$, distances along the curves $\gamma|_{[t,t']}$ and $\gamma$ coincide: for $s,s'\in[t,t']$, 
\[
  d_\gamma(s,s')=\min(\length(\gamma|_{[s,s']}), \length(\gamma)-\length(\gamma|_{[s,s']}))=\length(\gamma|_{[s,s']})=d_{\gamma|_{[t,t']}}(s,s')
  .
\]
Thus \eqref{gammatt'lsi} shows that $\gamma|_{[t,t']}$ is $(\delta,\epsilon)$-l.s.i.

Let $\gamma|_{[u,u']}$ be the subcurve of $\gamma|_{[t,t']}$ formed by concatenating those geodesic edges $\gamma_{[s_{j-1},s_j]}$ for which $[s_{j-1},s_j]\subset[t,t']$, possibly excluding up to two partial edges of $\gamma$ that overlap the endpoints of $[t,t']$.
Since $\gamma$ is a $(\delta,\epsilon,n)$-curve it follows that $\gamma|_{[u,u']}$ is a $(\delta,\epsilon,n)$-curve, and since $\gamma|_{[t,t']}$ is $(\delta,\epsilon)$-l.s.i.\ it follows that $\gamma|_{[u,u']}$ is too.
By \Cref{full-ball-lem}, $\Morrey{\gamma|_{[u,u']}}\leq 4\epsilon^{-1}+2n+4$.
Since $g$ consists of $\gamma|_{[u,u']}$, up to two partial geodesic edges of $\gamma$ that belong to $\gamma|_{[t,t']}$ but not $\gamma|_{[u,u']}$, and the geodesic $G_{\gamma(t'),\gamma(t)}$, \Cref{MorreyIsNorm,ball-growth-geodesic} give 
\begin{align*}
\Morrey{g}&\leq \Morrey{\gamma|_{[u,u']}} + 2\times 3 \\
&\leq 4\epsilon^{-1}+2n+10,
\end{align*}
which is the claimed inequality \eqref{MorreyTypeI}.
\end{proof}

\subsection{Type II cut operation}\label{ss-TypeII}

In a Type II cut operation, the input is a curve that is not a $(\delta,\epsilon,n)$-curve, which is split into one curve from which at least $n$ geodesic edges of length less than $\delta$ have been removed, along with one well-behaved curve satisfying a ball growth condition.
Note that repeatedly performing a Type II cut operation will take any piecewise geodesic curve into a $(\delta,\epsilon,n)$-curve with a finite number of iterations.

\begin{lemma}[Type II cut operation]\label{cut_II}
  Let $\epsilon \in (0,1)$ and $n\in\mathbb{N}$.  There exists
  \begin{align*}
    C_{\mathrm{II}}(\gamma) &: \Xgeo \setminus \Xden(\delta,\epsilon,n) \to \Xgeo \times \Xgeo
  \end{align*}
  with the following properties.
  Abbreviate $(\gamma',g):=C_{\mathrm{II}}(\gamma)$.
  Then
  \begin{align}\label{SameSumTypeII}
    [[\gamma]] = [[\gamma']] + [[g]],
  \end{align}
  where $g$ satisfies
  \begin{align}\label{MorreyTypeII}
    \Morrey{g} \leq 2(n+2\epsilon^{-1}+2),
  \end{align}
  the curves $\gamma',g$ satisfy the length bounds
  \begin{align}
    \label{LengthBoundTypeIIgamma}
    \length(\gamma') &\leq \length(\gamma)
    , \\
    \length(\gamma') +\length(g) &\leq \length(\gamma) + 4\epsilon^{-1}\delta
    \label{LengthBoundTypeIItotal}
    ,
  \end{align}
  and the number of small edges in the curve $\gamma'$ has the bound
  \begin{align}\label{SmallCountBoundTypeII}
    \nsmallpieces(\gamma',\delta) \leq \nsmallpieces(\gamma,\delta) -n
    .
  \end{align}
\end{lemma}

\begin{proof}
Let $(s_0,\dotsc,s_k)$ be a resolution of $\gamma$ into geodesic edges that realizes the minimum in \Cref{nsmallpieces-def}.
Since $\gamma$ is not a $(\delta,\epsilon,n)$-curve, there exists an interval $[s,s']$ such that $\gamma|_{[s,s']}$ has length at most $2\epsilon^{-1}\delta$ and contains at least $n+1$ edges of length smaller than $\delta$.
As in the comment following \Cref{delta-def}, by shrinking $[s,s']$ if necessary, we may assume that $s,s'\in\set{s_0,\dotsc,s_k}$ are endpoints of geodesic edges.
Shrinking $[s,s']$ further if necessary, we may further assume that $\gamma|_{[s,s']}$ has exactly $n+1$ edges of length smaller than $\delta$.
As in the proof of \Cref{cut_I}, we choose $t=t(\gamma)$ and $t'=t'(\gamma)$ in some definite way, for instance the lexicographically-smallest pair of endpoints of geodesic edges with the properties above.
Then we set $C_{\mathrm{II}}(\gamma)=C(\gamma,t(\gamma),t'(\gamma))=(\gamma',g)$, so that \eqref{SameSumTypeII} follows as a property of the basic cut operation.

The curve $\gamma'$ is formed from $\gamma$ by replacing part of $\gamma$ by a single geodesic edge.
By assumption this geodesic edge is length-minimizing, so $\length(\gamma')\leq\length(\gamma)$ follows, verifying \eqref{LengthBoundTypeIIgamma}.
In the construction of $\gamma'$, the removed part $\gamma|_{[t,t']}$ contains $n+1$ geodesic edges of length smaller than $\delta$ by assumption, and these are replaced by a single geodesic edge which may have length smaller than $\delta$, for an overall decrease of at least $n$ small geodesic edges, verifying \eqref{SmallCountBoundTypeII}.

Collectively, $\gamma'$ and $g$ contain all the edges of $\gamma$ plus two additional geodesic edges each of length $d(\gamma(t),\gamma(t'))\leq\length(\gamma|_{[t,t']})\leq 2\epsilon^{-1}\delta$; this proves \eqref{LengthBoundTypeIItotal}.

Note that $\gamma_{[t,t']}$ has length at most $2\epsilon^{-1}\delta$, so it contains at most $2\epsilon^{-1}$ edges of length at least $\delta$.
By construction, $\gamma_{[t,t']}$ contains exactly $n+1$ edges of length smaller than $\delta$.
Finally $g$ consists of these edges from $\gamma|_{[t,t']}$ together with one geodesic edge $G_{\gamma(t'),\gamma(t)}$.
Thus \Cref{C1-small-ball} with $k\leq 2\epsilon^{-1}+(n+1)+1$ gives \eqref{MorreyTypeII}.
\end{proof}

\subsection{Surgery algorithm and proof of \texorpdfstring{\Cref{surgery-lem}}{Lemma \ref*{surgery-lem}}}\label{ss-SurgeryAlg}

We are now ready to prove the Surgery Lemma.

\begin{proof}[Proof of \Cref{surgery-lem}]
Let $\gamma\in\Xgeo_c$ be a closed piecewise-geodesic curve. 
Fix a resolution $(s_0,\dotsc,s_k)$ of $\gamma$ into geodesic edges and set $\delta = \min\limits_{1 \leq j \leq k}  \length(\gamma|_{[s_{j-1},s_j]})$.
By construction, $m(\gamma,\delta)=0$ and therefore $\gamma$ is a $(\delta,\epsilon,n)$-curve.

The curves $\{g_i\}_{i=1}^N$ are produced using an iterative procedure, which we call the \emph{surgery algorithm}.
The iteration is initialized with $\gamma^{(0)} = \gamma$.
At the start of step $i\geq 1$, we have an input curve $\gamma^{(i-1)}$ to be processed and previously output curves $g_1,\dotsc,g_{i-1}$.
The iteration proceeds in two phases.

In the first phase, check whether $\gamma^{(i-1)}$ is a $(\delta,\epsilon,n)$-curve.
If it is not, perform a Type II cut as in \Cref{cut_II} to produce an output curve $g_i$ and a new curve $\gamma^{(i)}$ to be processed at the next step.
The output curve satisfies a ball growth condition by \eqref{MorreyTypeII}.
Meanwhile, \eqref{LengthBoundTypeIIgamma} and \eqref{SmallCountBoundTypeII} shows that $\gamma^{(i)}$ has no additional length and strictly fewer small geodesic edges compared to $\gamma^{(i-1)}$.
Thus performing a sufficient number of Type II cuts will yield a curve $\gamma^{(i')}$ ($i'\geq i$) that is a $(\delta,\epsilon,n)$-curve.

In the second phase, we may therefore assume that, perhaps after having performed several Type II cuts and having increased $i$, the curve $\gamma^{(i-1)}$ is a $(\delta,\epsilon,n)$-curve.
Next, check whether $\gamma^{(i-1)}$ is also $(\delta,\epsilon)$-l.s.i.
If it is, then \Cref{full-ball-lem} applies and yields a ball growth condition for $\gamma^{(i-1)}$, whereupon the algorithm terminates with $N=i$ and $g_N=\gamma^{(i-1)}$ as the final output curve.
If it is not, perform a Type I cut as in \Cref{cut_I}, again producing an output curve $g_i$ with a ball growth condition \eqref{MorreyTypeI} and a new curve $\gamma^{(i)}$.
By \eqref{LengthBoundTypeIgamma} and \eqref{SmallCountBoundTypeI},  $\gamma^{(i)}$ has strictly smaller length and only a moderate number of additional small geodesic edges compared to $\gamma^{(i-1)}$.
As we shall explain, it follows that this iterative procedure, which we set out formally in \Cref{SurgeryAlg}, must eventually terminate.

Note that $\gamma^{(0)}$ is a $(\delta,\epsilon,0)$-curve by our choice of $\delta$.
It follows that no Type II cuts can occur in the first $n/3$ iterations.

\begin{breakablealgorithm}\label{SurgeryAlg}
\caption{The Surgery Algorithm}
\begin{algorithmic}[1]
\Statex \textbf{Input}: A closed piecewise-geodesic curve $\gamma$ such that all geodesic edges in the curve have length at least $\delta$.
\Statex \textbf{Output}: A number $N\in\mathbb{N}$ and a finite sequence $(g_i)_{i=1}^N$ of closed piecewise-geodesic curves satisfying the conclusions of \Cref{surgery-lem}.
\State $i= 1$;
\State $j= 0$;
\State $T_1 = 0$;
\State $T_2 = 0$;
\State $\gamma^{(0)}=\gamma$;
\While{True}
\While{$\gamma^{(i-1)}$ is not a $(\delta,\epsilon,n)$-curve}
  \State Apply a Type II cut $(\gamma',g):=C_{\mathrm{II}}(\gamma^{(i-1)})$;
  \State $g_i=g$; \Comment{By \Cref{cut_II}, $g_i$ satisfies a ball growth condition.}
  \State $\gamma^{(i)}=\gamma'$;
  \State $i\leftarrow i+1$;
  \State $T_2 \leftarrow T_2+1$;
\EndWhile
\If{$\gamma^{(i-1)}$ is $(\delta,\epsilon)$-l.s.i.}
\State $N := i$;
\State $g_N = \gamma^{(i-1)}$; \Comment{By \Cref{full-ball-lem}, $g_N$ satisfies a ball growth condition.}
\State break;
\Else
  \State Apply a Type I cut $(\gamma',g):=C_{\mathrm{I}}(\gamma^{(i-1)})$ and $\beta = \beta(\gamma^{(i-1)})$.
  \State $g_i=g$; \Comment{By \Cref{cut_I}, $g_i$ satisfies a ball growth condition.}
  \State $\gamma^{(i)}=\gamma'$;
    \State $j\leftarrow j+1$;
  \State $\beta_j = \beta$;
  \State $i\leftarrow i+1$;
  \State $T_1 \leftarrow T_1+1$;
  \EndIf
\EndWhile
\State \Return the number $N$ and the finite sequence $(g_i)_{i=1}^N$
\end{algorithmic}
\end{breakablealgorithm}

Consider any step $i<N$ at which the algorithm has not yet terminated.
Write $T_\mathrm{I}^{(i)}$ for the number of Type I cuts performed up to and including step $i$.
By \eqref{LengthBoundTypeIgamma}, \eqref{LengthBoundTypeIIgamma}, and the lower bound $\beta(\gamma)\geq\delta$ from \Cref{cut_I},
\begin{align}
  \length(\gamma^{(i)}) 
  &\leq \length(\gamma) - \sum_{j=1}^{T_\mathrm{I}^{(i)}} (1-\epsilon)\beta_j 
  \label{LgammabetajBound}
  \\
  &\leq \length(\gamma) - (1-\epsilon)\delta T_\mathrm{I}^{(i)}
  \notag 
  .
\end{align}
In particular, the total number $T_\mathrm{I}$ of Type I cuts is bounded by
\begin{equation*}
  T_\mathrm{I} \leq \frac{\length(\gamma)}{(1-\epsilon)\delta}
  .
\end{equation*}
Similarly, writing $T_\mathrm{II}^{(i)}$ for the number of Type II cuts by step $i$ and using \eqref{SmallCountBoundTypeI}, \eqref{SmallCountBoundTypeII}, and $\nsmallpieces(\gamma^{(0)},\delta)=\nsmallpieces(\gamma,\delta)=0$,
\begin{equation*}
  \nsmallpieces(\gamma^{(i)},\delta) \leq 3 T_\mathrm{I}^{(i)} - n T_\mathrm{II}^{(i)}
\end{equation*}
so that the total number $T_\mathrm{II}$ of Type II cuts is bounded by
\begin{equation}\label{T2Bound}
  T_\mathrm{II} \leq \frac{3}{n} T_\mathrm{I} \leq \frac{3}{n}\frac{\length(\gamma)}{(1-\epsilon)\delta}
  .
\end{equation}
In particular, the total number of cuts is bounded and the algorithm therefore terminates in a finite number $N=T_\mathrm{I}+T_\mathrm{II}+1$ of steps.

By \eqref{SameSumTypeI} and \eqref{SameSumTypeII}, both Type I and Type II cuts lead to $[[\gamma^{(i-1)}]] = [[\gamma^{(i)}]] + [[g_i]]$, and it follows by induction that
\begin{equation}
  [[\gamma]] = [[\gamma^{(i)}]] + \sum_{k=1}^i [[g_k]]
\end{equation}
for $i<N$.
Part~\ref{item:SurgeryMeasureIdentity} of \Cref{surgery-lem} follows by taking $i=N-1$ and recalling that $g_N=\gamma^{(N-1)}$.

For part~\ref{item:SurgeryTotalLength}, note from \eqref{LengthBoundTypeItotal} that the $j^\text{th}$ Type I cut increases the total length of the curves by at most $2\epsilon\beta_j$, while from \eqref{LengthBoundTypeIItotal} each Type II cut increases the total length by at most $4\epsilon^{-1}\delta$.
Thus
\begin{align*}
  \sum_{i=1}^N \length(g_i) &\leq \length(\gamma) + 2\epsilon\sum_{j=1}^{T_\mathrm{I}}\beta_j + 4\epsilon^{-1}\delta T_\mathrm{II}
  .
\end{align*}
From \eqref{LgammabetajBound} we have $\sum_{j=1}^{T_1} \beta_j \leq \length(\gamma)/(1-\epsilon)$.
Combining with \eqref{T2Bound},
\begin{align*}
  \sum_{i=1}^N \length(g_i) &\leq \length(\gamma) + \frac{2\epsilon \length(\gamma)}{1-\epsilon} + \frac{12\epsilon^{-1}\delta \length(\gamma)}{n(1-\epsilon)\delta}
  ,
\end{align*}
and simplifying proves part~\ref{item:SurgeryTotalLength}.

Finally part~\ref{item:SurgeryBallGrowth} follows by applying \eqref{MorreyTypeI}, \eqref{MorreyTypeII} or \Cref{full-ball-lem}, depending on whether $g_i$ was produced from a Type I cut, from a Type II cut, or from the final step of the algorithm, respectively.
\end{proof}

\section{Metric currents without boundary and closed curves}\label{CurrentsAndCurves}

\subsection{Differential forms and metric currents in \texorpdfstring{$E$}{E}}\label{differentialforms}
We here recall a few relevant notations related to the space of $k$-dimensional forms in the metric space.
In the sequel we follow the conventions of \cite{AmbrosioKirchheim}.
Throughout this paper $E$ (unless explicitly stated otherwise) denotes a complete, separable, geodesic metric space and $\operatorname{Lip}(E, \mathbb{R})$ (resp.\ $\operatorname{Lip_b}(E,\mathbb{R}))$ denotes the set of all (resp.\ bounded) Lipschitz maps $f: E \to \mathbb{R}$.
We also denote by
\begin{align*}\label{massmeasure}
  \mathcal{D}^k (E) : = \{ (f, \pi_1, \dots , \pi_k)  \colon f\in \operatorname*{Lip_b}(E, \mathbb{R}), \pi_1,\dotsc,\pi_k \in  \operatorname*{Lip}(E, \mathbb{R}) \}
\end{align*}
the space of $k$-dimensional differential forms on $E$ and by $\mathcal{M}_k(E)$, its ``dual'', the set of all real $k$-dimensional metric current of $E$ in the sense of of Ambrosio and Kirchheim \cite{AmbrosioKirchheim}.

In particular, for every $ T \in \mathcal{M}_k(E)$, there exists a finite Borel measure $\mu $ over $E$ such that
\begin{align}
  |T(f, \pi_1, \dots , \pi_k)| \leq \prod^k_{i=1} \operatorname*{Lip}(\pi_i) \int_E |f| \, d \mu
\end{align}
holds for every $(f, \pi_1, \dots , \pi_k) \in \mathcal{D}^k (E)$, where $\Lip(\pi_i)$ denotes the Lipschitz constant of the function $\pi_i$.
We define the mass measure $\mu_T$ of $T$ to be the minimum over all the finite positive Borel measure $\mu$ satisfying \eqref{massmeasure}.
Moreover, the total mass $\mathbb{M}(T)$ of $T$ is defined to be $\mu_T(E)$.

Given a $k$-dimensional current $T$, we denote by $\partial T$ its boundary defined by the formula
\begin{align*}
  \partial T(\omega):= T(d\omega)
\end{align*}
for all $\omega \in \mathcal{D}^{k-1}(E)$.
Here, $d:\mathcal{D}^{k-1}(E) \to \mathcal{D}^{k}(E)$ denotes the metric exterior derivative operator given by
\begin{align*}
  d(f , \pi_1, \pi_2, \dotsc, \pi_{k-1}) := (1 , f , \pi_1, \pi_2  ,\dotsc,\pi_{k-1})
  .
\end{align*}

We can associate to each Lipschitz curve $\gamma\in\Theta(E)$ a one-dimensional normal current $[[\gamma]]\in\mathcal{M}_1(E)$ defined as follows.
For all 1-dimensional differential forms $\omega=(f, \pi)\in\mathcal{D}^1(E)$, we define the scalar $[[\gamma]](\omega)$ to be the Riemann-Stieltjes integral
\begin{align}\label{currentcurve}
  [[\gamma]] (f , \pi) : = \int_0^{\length(\gamma)} (f \circ \gamma)(t) \; d (\pi \circ \gamma)(t)
  .
\end{align}
Note that if $\gamma$ is a concatenation of curves $\gamma_1,\dotsc,\gamma_k$ then we may split the Riemann-Stieltjes integral \eqref{currentcurve} across the subintervals corresponding to the concatenated curves, so that $[[\gamma]]=[[\gamma_1]]+\dotsb+[[\gamma_k]]$ follows immediately.
This verifies the remaining assertion of \Cref{MorreyIsNorm}.

We note that the Riemann-Stieltjes integral \eqref{currentcurve} is unchanged by increasing reparametrizations of $\gamma$, so the mapping $\gamma\mapsto [[\gamma]]$ can also be interpreted as a mapping on $\tilde{\Theta}(E)$.
On the other hand, if the curve $\tilde{\gamma}$ is the curve $\gamma$ traversed in reverse order then $[[\tilde{\gamma}]](\omega)=-[[\gamma]](\omega)$ for all $\omega\in \mathcal{D}^1(E)$, or more succinctly $[[\gamma]]+[[\tilde{\gamma}]]=0$ as asserted in \Cref{MainResultSubsect} and \eqref{ReversedGeodesics}.

From a Lipschitz curve $\gamma\in\Theta(E)$ parametrized by arc length, we can form two different measures on $E$: $\mu_\gamma$, the image of length measure on $[0,\length(\gamma)]$ under the mapping $\gamma\colon[0,\length(\gamma)]\to E$; and $\mu_{[[\gamma]]}$, the mass measure of the $1$-current $[[\gamma]]\in\mathcal{M}_1(E)$.

\begin{lemma}\label{easyestimateformass}
  Let $\gamma\in\Theta(E)$.
  Then $\mu_{[[\gamma]]}\leq\mu_\gamma$, i.e., $\mu_{[[\gamma]]} (B) \leq  \mu_{\gamma}(B)$ for all Borel sets $B$.
  In particular, $\mathbb{M}([[\gamma]])\leq \length(\gamma)$ and
  \begin{equation*}
    \abs{[[\gamma]] (f, \pi)} \leq \norm{f}_\infty \Lip(\pi) \length(\gamma)
    .
  \end{equation*}
\end{lemma}

It can be verified that if $\gamma$ is injective, the measures $\mu_{[[\gamma]]}$ and $\mu_\gamma$ coincide, and therefore the total mass $\mathbb{M}([[\gamma]]) = \mu_{[[\gamma]]}(E) $ of the current $[[\gamma]]$ equals the length $\length(\gamma)$.
In general, $\mu_{[[\gamma]]}$ and $\mu_{\gamma}$ may differ, for instance when the curve $\gamma$ traverses the same segment twice in opposite directions.

\begin{proof}
Let $f\in\Lipb(E,\mathbb{R}),\pi\in\Lip(E,\mathbb{R})$ be arbitrary.
Since $\gamma$ is parametrized by arc length, the composition $\pi\circ\gamma\colon\mathbb{R}\to\mathbb{R}$ has Lipschitz constant bounded by $\Lip(\pi)$.
Hence the Riemann-Stieltjes integral \eqref{currentcurve} can be bounded by
\begin{align*}
  \bigl| [[\gamma]] (f, \pi) \bigr| 
  \leq \Lip(\pi) \int_0^{\length(\gamma)} |f(\gamma(t))| \, dt
  = \Lip(\pi) \int_E |f|\, d \mu_\gamma
  ,
\end{align*}
where the last inequality is by the definition of $\mu_\gamma$. 
Thus \eqref{massmeasure} holds with $k=1$, $T=[[\gamma]]$ and $\mu=\mu_\gamma$, and hence $\mu_{[[\gamma]]}\leq \mu_\gamma$ as measures by the minimality of $\mu_{[[\gamma]]}$.
The remaining bounds follow by considering the total masses of these measures.
\end{proof}

For Lipschitz curves $\gamma_1:[a_1,b_1] \to E$ and $\gamma_2: [a_2,b_2] \to E$, define
\begin{multline*}
  d_\Theta (\gamma_1, \gamma_2) 
  \\
  : = \inf \{   \max_{t \in [a_1,b_1]} d(\gamma_1(t), \gamma_2(\phi(t))) \colon \phi\colon[a_1,b_1] \rightarrow [a_2,b_2]\text{ bijective increasing}  \}
  .
\end{multline*}
Then $d_\Theta$ is a pseudometric on the space of all Lipschitz curves and induces a metric on $\tilde{\Theta}(E)$, the space of equivalence classes of Lipschitz curves modulo increasing reparametrization.
Equivalently, the restriction $d_\Theta\vert_{\Theta(E)}$ gives a metric on $\Theta(E)$ directly.
We equip $\Theta(E)$ with the associated Borel $\sigma$-algebra.
We also define
\begin{align*}
  \Theta_l(E): = \{ \gamma \in \Theta(E): \mathbb{M} ([[\gamma]]) = l \}
\end{align*}
and we write $b(\gamma)=\gamma(0)$, $e(\gamma)=\gamma(\length(\gamma))$ for the beginning and ending points, respectively, of the curve $\gamma\in\Theta(E)$.

In \cite{PS}*{Theorem~3.1}, Paolini and Stepanov have shown that if $E$ is a complete, separable metric space and $T \in \mathcal{M}_1(E)$ satisfies $\partial T =0 $, then there exists a positive Borel measure $\overline{\eta}_1 $ on $\Theta_1(E)$ such that
\begin{align}\label{PaoT}
  T(\omega) = \int_{\Theta_1(E)} [[\gamma]] (\omega)\, d\overline{\eta}_1(\gamma)
\end{align}
holds for every $\omega \in \mathcal{D}^1 (E)$ and $\mathbb{M}([[\gamma]]) = \length(\gamma) = 1$ for $\overline{\eta}_1$-a.e $\gamma \in \Theta_1(E) $. Moreover, the mass measure $\mu_T$ of $T$ satisfies
\begin{align}\label{Paophi}
  \int_E \phi(x) \, d \mu_T(x) = \int_{\Theta_1(E)} \phi(b(\gamma)) \, d \overline{\eta}_1(\gamma) = \int_{\Theta_1(E)} \phi(e(\gamma)) \, d \overline{\eta}_1(\gamma)
\end{align}
for every measurable non-negative function $\phi\colon E\to\cointerval{0,\infty}$, and the total mass $\mathbb{M}(T)$ of $T$ is given by $\mathbb{M}(T) = \overline{\eta}_1(\Theta_1(E))$.

Paolini and Stepanov show that this measure on length-1 curves can be concatenated to form a measure on bi-infinite curves \cite{PS}*{Proposition 4.2}.
We will not use the full strength of this statement, but their construction yields as a by-product a sequence of measure $\overline{\eta}_l$ on curves of integer length $l\in\mathbb{N}$ satisfying
\begin{gather}
  \label{weaklengthl}
  T(\omega) = \frac{1}{l} \int_{\Theta_l(E)} [[\gamma]] (\omega) \, d\overline{\eta}_l (\gamma)
  \quad\text{for every }\omega \in \mathcal{D}^1 (E)
  ,
  \\
  \label{lengthlmassmeasure}
  \int_E \phi(x) \, d \mu_T(x) = \int_{\Theta_l(E)} \phi(b(\gamma)) \, d \overline{\eta}_l(\gamma) = \int_{\Theta_l(E)} \phi(e(\gamma)) \, d \overline{\eta}_l(\gamma)
  ,
  \\
  \label{lengthlExactly}
  \mathbb{M}([[\gamma]]) = \length(\gamma ) = l \quad\text{for $\overline{\eta}_l$-a.e.\ }\gamma \in \Theta_l (E)
  .
\end{gather}
In the notation of \cite{PS}*{Proposition 4.2}, the measure $\overline{\eta}_l$ is obtained as the image measure of $\hat{\eta}$ on $C(\mathbb{R},E)$ under the mapping $\gamma\in C(\mathbb{R},E)\mapsto \gamma\vert_{[0,l]}$, with $\gamma\vert_{[0,l]}$ interpreted as an element of $\Theta(E)$.
Modulo differences in notation, the identity \eqref{weaklengthl} is equivalent to \cite{PS}*{Remark~4.3} with $(m,n)=(l,0)$.
According to \cite{PS}*{Proposition~4.2(b)}, the marginal distribution under $\overline{\eta}_l$ of each subcurve $\gamma\vert_{[j-1,j]}$, $j\in\set{1,\dotsc,l}$, coincides with $\overline{\eta}_1$.
Thus \eqref{lengthlExactly} follows because $\overline{\eta}_l$ is supported on curves that are the concatenation of $l$ subcurves each of length 1.
Similarly, \eqref{lengthlmassmeasure} follows by considering $j=1$ and $j=l$,  whereupon \eqref{Paophi} shows that both $b(\gamma\vert_{[0,1]})$ and $e(\gamma\vert_{[l-1,l]})$ follow the distribution $\mu_T$.

We next show how the preceding results of Paolini and Stepanov can be used to decompose metric currents without boundary in terms of closed curves.  
This is a generalization of an assertion of Bourgain and Brezis in the Euclidean context \cites{BourgainBrezis2004,BourgainBrezis2007}.

\begin{theorem}\label{bbassertion}
  Let $E$ be a complete, separable, geodesic metric space.
  Then for $T \in \mathcal{M}_1(E)$ satisfying $T\neq 0$ and $\partial T=0$, there exists a sequence of positive integers $(n_l)_{l=1}^\infty$ tending to $\infty$ and a countable collection of oriented piecewise-geodesic closed curves $\hat{\gamma}_{i,l}$, $l\in\mathbb{N},i\in\set{1,\dotsc,n_l}$, such that $\mathbb{M}(\hat{\gamma}_{i,l}) \leq \length(\hat{\gamma}_{i,l}) \leq 2 l$ for all $i,l$,
  \begin{align} \label{closed1}
    T(\omega)= \lim_{l \to \infty} \frac{\mathbb{M}(T)}{n_l \cdot l} \sum_{i=1}^{n_l}  [[\hat{\gamma}_{i,l} ]](\omega)
  \end{align}
  for all $\omega \in \mathcal{D}^1(E)$, and
  \begin{align} \label{closed2}
    \lim_{l \to \infty} \frac{1}{n_l \cdot l} \sum_{i=1}^{n_l} \mathbb{M}([[\hat{\gamma}_{i,l}]]) = \lim_{l \to \infty} \frac{1}{n_l \cdot l} \sum_{i=1}^{n_l} \length(\hat{\gamma}_{i,l}) =1
    .
  \end{align}
\end{theorem}

The proof of \autoref{bbassertion} is based on the use of the strong law of large numbers to obtain initial curves, approximating these curves by piecewise-geodesic interpolations, and finally adding segments to close the curves.  To this end it will be useful to recall the following definition and a lemma concerning the convergence of piecewise-geodesic interpolations of a current.

\begin{definition}\label{sampling}
  Let $\gamma\in\Theta(E)$ be a non-constant curve parametrized by arc length, and let $0<\delta<\length(\gamma)$.
  Set $k=\lceil\length(\gamma)/\delta\rceil$ and set $s_j=\min(j\delta,\length(\gamma))$ for $j=0,\dotsc,k$.
  We define $\gamma^\delta$, the piecewise-geodesic sampling of $\gamma$ at scale $\delta$, to be the concatenation of the geodesics $G_{\gamma(s_{j-1}),\gamma(s_j)}$ for $j=1,\dotsc,k$.
\end{definition}

\begin{lemma}\label{GeodesicSampledCurrent}
  Let $\gamma\in\Theta(E)$ and let $0< \delta < \length(\gamma)$.
  Then $\length(\gamma^\delta)\leq\length(\gamma)$ and
  \begin{equation*}
    \lim_{\delta \to 0}  [[\gamma^\delta]](\omega)= [[\gamma]](\omega)
  \end{equation*}
  for every $\omega \in \mathcal{D}^1 (E)$.
\end{lemma}

\begin{proof}
Each of the subcurves $\gamma_{[s_{j-1},s_j]}$ has length at most $\delta$ by construction.
Because the geodesics $G_{\gamma(s_{j-1}),\gamma(s_j)}$ are length-minimizing, $\length(\gamma^\delta) \leq \length(\gamma)$ for all $\delta>0$.
Moreover, each of the geodesic edges of $\gamma^\delta$ has length at most $\delta$ and has the same endpoints as the corresponding subcurve $\gamma_{[s_{j-1},s_j]}$.
It follows that 
\begin{equation*}
  d_\Theta(G_{\gamma(s_{j-1}),\gamma(s_j)},\gamma_{[s_{j-1},s_j]})\leq 2\delta
\end{equation*}
and consequently $d_\Theta(\gamma^\delta,\gamma)\leq 2\delta$.
In particular, $\gamma^\delta \to \gamma$ in $\Theta(E)$ as $\delta\to 0$.
Hence we may apply \cite{PS_2012}*{Lemma~4.1} to conclude that $[[\gamma^\delta]](\omega)\to [[\gamma]](\omega)$ as $\delta\to 0$, for all $\omega \in \mathcal{D}^1 (E)$.
\end{proof}

We now prove \autoref{bbassertion}.
\begin{proof}[Proof of \autoref{bbassertion}]
Fix $T \in \mathcal{M}_1{(E)}$ satisfying $T\neq 0$ and $\partial T=0$.
Since $T\neq 0$, by scaling we may assume $\mathbb{M}(T)=1$.

In order to apply appropriate diagonal arguments, we begin by reducing to a countable collection of limits.
Let $\set{\tilde{\pi}_j}_{j\in\mathbb{N}}$ be the countable dense collection of bounded Lipschitz functions given in \Cref{Lipschitz_separable}.
To make the notation more consistent, we let $\{\tilde{f}_k\}_{k\in\mathbb{N}}$ denote a copy of the same countable dense collection with a different indexing variable.
Fix an arbitrary reference point $x_0\in E$.
For $r\in\mathbb{Q}\cap\cointerval{0,\infty}$, let $\phi_{r,b},\phi_{r,e}\colon\Theta(E)\to[0,1]$ be the indicator functions
\begin{equation}\label{EndpointCutoffs}
  \phi_{r,b}(\gamma)=\begin{cases}
    1&\text{if }d(x_0,b(\gamma))\geq r,\\
    0&\text{otherwise},
  \end{cases}
  \quad
  \phi_{r,e}(\gamma)=\begin{cases}
    1&\text{if }d(x_0,e(\gamma))\geq r,\\
    0&\text{otherwise}.
  \end{cases}
\end{equation}

With these preparations, the functions
\begin{equation*}
  \gamma\mapsto [[\gamma]](\tilde{f}_k,\tilde{\pi}_j)
  ,\quad j,k\in\mathbb{N}
  ,
\end{equation*}
and the functions $\phi_{r,b},\phi_{r,e}$ over all $r\in\mathbb{Q}\cap\cointerval{0,\infty}$, form a countable collection of real-valued measurable functions on $\Theta(E)$, each of which is integrable with respect to the measures $\overline{\eta}_l$ from \eqref{weaklengthl}--\eqref{lengthlExactly}, for all $l\in\mathbb{N}$.
The same is true for the functions $\psi_{m,Q,\epsilon,C},\tilde{\psi}_{m,Q,\epsilon,C}$ from \Cref{pi_bounds}.
Throughout, the functions $\psi_{m,Q,\epsilon,C},\tilde{\psi}_{m,Q,\epsilon,C}$ are countably indexed over all $m\in\mathbb{N}$; all finite subsets $Q\subset E_\mathbb{Q}$, where $E_\mathbb{Q}$ denotes a fixed countable dense subset of the separable metric space $E$; and all $\epsilon,C\in\mathbb{Q}\cap\cointerval{0,\infty}$.
For brevity we will simply say ``all admissible choices of $m,Q,\epsilon,C$'' in what follows.

Following the same argument as in \cite{GHS}*{Theorem~2.1}, the Strong Law of Large Numbers allows us to deduce from Paolini and Stepanov's identities \eqref{weaklengthl}--\eqref{lengthlExactly} that there exist countable collections of Lipschitz curves $\gamma_{i,l}\in\Theta(E)$, $i,l\in\mathbb{N}$, such that
\begin{gather}
  \label{countable_identity}
  T(\tilde{f}_k,\tilde{\pi}_j) = \lim\limits_{ n \to \infty} \frac{1}{n \cdot l} \sum^{n}_{i = 1} [[\gamma_{i,l}]] (\tilde{f}_k,\tilde{\pi}_j)
  \quad\text{for all }j,k,l\in\mathbb{N}
  ,
  \\
  \label{stronglawmassmeasure}
  \mu_T(\set{x\in E\colon d(x,x_0)\geq r}) = \lim_{ n \to \infty} \frac{1}{n} \sum_{i=1}^n \phi_{r,b}(\gamma_{i,l}) = \lim_{ n \to \infty} \frac{1}{n} \sum_{i=1}^n \phi_{r,e}(\gamma_{i,l})
\end{gather}
for all $r\in\mathbb{Q}\cap\cointerval{0,\infty}$ and all $l\in\mathbb{N}$, and
\begin{gather}
  \label{stronglawlength}
  \mathbb{M}([[\gamma_{i,l}]]) = \length(\gamma_{i,l}) = l
  \quad\text{for all }i,l\in\mathbb{N}
  ,
\end{gather}
and also such that
\begin{equation}\label{stronglawpsi}
  \begin{aligned}
  \int_{\Theta_l(E)}  \psi_{m\cdot l,Q,\epsilon,C}(\gamma)  \, d\overline{\eta}_l (\gamma) &= \lim\limits_{ n \to \infty} \frac{1}{n} \sum^{n}_{i = 1}    \psi_{m\cdot l,Q,\epsilon,C}(\gamma_{i,l})
  ,
  \\
  \int_{\Theta_l(E)}  \tilde{\psi}_{m\cdot l,Q,\epsilon,C}(\gamma)  \, d\overline{\eta}_l (\gamma) &= \lim\limits_{ n \to \infty} \frac{1}{n} \sum^{n}_{i = 1}    \tilde{\psi}_{m\cdot l,Q,\epsilon,C}(\gamma_{i,l})
\end{aligned}
\end{equation}
hold for all $l\in\mathbb{N}$ and for all admissible choices of $m,Q,\epsilon,C$.

By \Cref{GeodesicSampledCurrent}, replacing $\gamma_{i,l}$ by the piecewise-geodesic sampling $\gamma^\delta_{i,l}$ and taking $\delta\to 0$ leaves \eqref{countable_identity} unchanged:
\begin{equation*}
  T(\tilde{f}_k,\tilde{\pi}_j) = \lim_{n \to \infty} \lim_{\delta\to 0} \frac{1}{n \cdot l} \sum^{n}_{i = 1} [[\gamma^\delta_{i,l}]] (\tilde{f}_k,\tilde{\pi}_j)
  \quad\text{for all }j,k,l\in\mathbb{N}
  .
\end{equation*}
An analogous result holds for \eqref{stronglawpsi} because of \Cref{psiSampling}.
Meanwhile the property \eqref{stronglawmassmeasure} is unaffected by piecewise-geodesic sampling since $\gamma_{i,l}$ and $\gamma^\delta_{i,l}$ have the same beginning and ending points.
Since \eqref{countable_identity} and \eqref{stronglawpsi} involve countably many limits, we may use a diagonal argument to find a sequence $\tilde{\delta}_n\to 0$ such that 
\begin{gather}
  T(\tilde{f}_k,\tilde{\pi}_j) = \lim_{n \to \infty} \frac{1}{n \cdot l} \sum^{n}_{i = 1} [[\gamma^{\tilde{\delta}_n}_{i,l}]] (\tilde{f}_k,\tilde{\pi}_j)
  ,
  \label{TidentityFixedl}
  \\
  \begin{aligned}
    \int_{\Theta_l(E)} \frac{1}{l}\psi_{m\cdot l,Q,\epsilon,C}(\gamma) \, d\overline{\eta}_l (\gamma) &= \lim\limits_{ n \to \infty} \frac{1}{n\cdot l} \sum^{n}_{i = 1}    \psi_{m\cdot l,Q,\epsilon,C}(\gamma^{\tilde{\delta}_n}_{i,l})
    ,
    \\
    \int_{\Theta_l(E)} \frac{1}{l}\tilde{\psi}_{m\cdot l,Q,\epsilon,C}(\gamma) \, d\overline{\eta}_l (\gamma) &= \lim\limits_{ n \to \infty} \frac{1}{n\cdot l} \sum^{n}_{i = 1}    \tilde{\psi}_{m\cdot l,Q,\epsilon,C}(\gamma^{\tilde{\delta}_n}_{i,l})
  \end{aligned}
  \label{psiIdentitiesFixedl}
\end{gather}
for all $j,k,l\in\mathbb{N}$ and for all admissible choices of $m,Q,\epsilon,C$.

By \Cref{etaRepeats}, the integrals in \eqref{psiIdentitiesFixedl} do not depend on $l$.
Hence we may apply a further diagonal argument to the countable many limits in  \eqref{TidentityFixedl}--\eqref{psiIdentitiesFixedl} and \eqref{stronglawmassmeasure} to obtain subsequences $n_l, l\in\mathbb{N}$ and $\delta_l := \tilde{\delta}_{n_l}$ such that $n_l\in\mathbb{N}$, $n_l\to\infty$ as $l\to\infty$ and
\begin{gather}
  T(\tilde{f}_k,\tilde{\pi}_j) = \lim_{l \to \infty} \frac{1}{n_l \cdot l} \sum_{i = 1}^{n_l} [[\gamma^{\delta_l}_{i,l}]] (\tilde{f}_k,\tilde{\pi}_j)
  \label{diag_T}
  ,
  \\
  \begin{aligned}
    \int_{\Theta_1(E)} \psi_{m,Q,\epsilon,C}(\gamma) \, d\overline{\eta}_1(\gamma) &= \lim_{ l \to \infty} \frac{1}{n_l\cdot l} \sum_{i = 1}^{n_l}    \psi_{m\cdot l,Q,\epsilon,C}(\gamma^{\delta_l}_{i,l})
    ,
    \\
    \int_{\Theta_1(E)} \tilde{\psi}_{m,Q,\epsilon,C}(\gamma) \, d\overline{\eta}_1 (\gamma) &= \lim_{ l \to \infty} \frac{1}{n_l\cdot l} \sum_{i = 1}^{n_l}    \tilde{\psi}_{m\cdot l,Q,\epsilon,C}(\gamma^{\delta_l}_{i,l})
    ,
  \end{aligned}
  \label{diag_psi}
  \\
  \mu_T(\set{x\in E\colon d(x,x_0)\geq r}) = \lim_{ l \to \infty} \frac{1}{n_l} \sum_{i=1}^{n_l} \phi_{r,b}(\gamma_{i,l}) = \lim_{ l \to \infty} \frac{1}{n_l} \sum_{i=1}^{n_l} \phi_{r,e}(\gamma_{i,l})
  \label{diag_phi}
\end{gather}
for all $j,k\in\mathbb{N}$, for all $r\in\mathbb{Q}\cap\cointerval{0,\infty}$, and for all admissible choices of $m,Q,\epsilon,C$.

We next show that \eqref{diag_T} holds with $\tilde{f}_k,\tilde{\pi}_j$ replaced by general arguments. 
By \Cref{Lipschitz_separable}, we may find subsequences $(f_k)_{k\in\mathbb{N}},(\pi_j)_{j\in\mathbb{N}}$ chosen from the countable families $\{\tilde{f}_k\}_{k \in \mathbb{N}}, \{\tilde{\pi}_j\}_{j \in \mathbb{N}}$ such that
\begin{align}
  f_k &\to f\text{ pointwise}, \quad \norm{f_k}_\infty\leq \norm{f}_\infty,\: \Lip(f_k)\leq\Lip(f),\: \Lip(f_k)\to \Lip(f)
  \label{f_converge}
  , 
  \\
  \pi_j &\to \pi\text{ pointwise}, \quad \Lip(\pi_j)\leq\Lip(\pi),\: \Lip(\pi_j) \to \Lip(\pi) 
  \label{pi_converge}
  .
\end{align}
Define
\begin{equation*}
  C:= \lceil\max\set{ \norm{f}_\infty,\Lip(f),\Lip(\pi)}\rceil + 1
  ,
\end{equation*}
so that by construction $C\in\mathbb{Q}\cap(0,\infty)$ and
\begin{equation}\label{BoundByC}
  \norm{f}_\infty,\norm{f_k},\Lip(f),\Lip(f_k),\Lip(\pi),\Lip(\pi_j)\leq C\text{ for all $j,k$}
  .
\end{equation}

We recall that by definition, a metric $1$-current $T$ is continuous with respect to the pointwise convergence \eqref{pi_converge}, while the continuity with respect to \eqref{f_converge} follows from \eqref{massmeasure} and Lebesgue's dominated convergence theorem, and therefore
\begin{equation}\label{TfpiTfkpij}
  T(f,\pi) = \lim_{k \to \infty} \lim_{j \to \infty} T(f_k,\pi_j) =    \lim_{j \to \infty} \lim_{k \to \infty} T(f_k,\pi_j)
  .
\end{equation}
Fixing $k,j$ for the moment, note that \eqref{diag_T} applies to $f_k,\pi_j$ and use multilinearity to find
\begin{align*}
  &\limsup_{l\to\infty} \abs{T(f_k,\pi_j) - \frac{1}{n_l\cdot l}\sum_{i=1}^{n_l} [[\gamma^{\delta_l}_{i,l}]](f,\pi)} 
  \\&\quad
  = \limsup_{l\to\infty}\abs{\frac{1}{n_l\cdot l}\sum_{i=1}^{n_l} \left( [[\gamma^{\delta_l}_{i,l}]](f_k,\pi_j) - [[\gamma^{\delta_l}_{i,l}]](f,\pi) \right)} 
  \\&\quad
  \leq \limsup_{ l \to \infty} \frac{1}{n_l \cdot l} \sum_{i = 1}^{n_l} \abs{[[\gamma^{\delta_l}_{i,l}]] (f_k-f,\pi_j)}
  + \limsup_{ l \to \infty} \frac{1}{n_l \cdot l} \sum_{i = 1}^{n_l} \abs{[[\gamma^{\delta_l}_{i,l}]] (f,\pi_j-\pi)}
  .
\end{align*}
Let $m\in\mathbb{N}$, let $Q\subset E_\mathbb{Q}$ be finite, and let $\epsilon\in\mathbb{Q}\cap(0,\infty)$.
By \eqref{f_converge}--\eqref{pi_converge} there exist $j_0(Q,\epsilon)$ and $k_0(Q,\epsilon)$ such that $\sup_{q\in Q}\abs{\pi_j(q)-\pi(q)}\leq\epsilon$ for all $j\geq j_0(Q,\epsilon)$ and $\sup_{q\in Q}\abs{f_k(q)-f(q)}\leq\epsilon$ for all $k\geq k_0(Q,\epsilon)$.
Together with \eqref{BoundByC}, we may therefore apply \Cref{pi_bounds} for this range of $j,k$ and with $m$ replaced by $m\cdot l$.
Thus
\begin{align*}
  &\limsup_{j,k,l\to\infty} \abs{T(f_k,\pi_j) - \frac{1}{n_l\cdot l}\sum_{i=1}^{n_l} [[\gamma^{\delta_l}_{i,l}]](f,\pi)}
  \\&\quad
  \leq \limsup_{j,k,l\to\infty} \frac{1}{n_l \cdot l} \sum_{i = 1}^{n_l} \tilde{\psi}_{m\cdot l,Q,\epsilon,C}(\gamma^{\delta_l}_{i,l})
  + \limsup_{j,k,l\to\infty} \frac{1}{n_l \cdot l} \sum_{i = 1}^{n_l} \psi_{m\cdot l,Q,\epsilon,C}(\gamma^{\delta_l}_{i,l})
  \\&\quad
  = \int_{\Theta_1(E)} \tilde{\psi}_{m,Q,\epsilon,C}(\gamma) \, d\overline{\eta}_1 (\gamma) + \int_{\Theta_1(E)} \psi_{m,Q,\epsilon,C}(\gamma) \, d\overline{\eta}_1 (\gamma)
  ,
\end{align*}
where the last equality uses \eqref{diag_psi}.
The left-hand side of this inequality does not depend on $m,Q,\epsilon$, so following \Cref{psi_vanishes} we may take $\epsilon\to 0$, $Q\nearrow E_\mathbb{Q}$, $m\to\infty$ in that order to obtain
\begin{equation*}
  \limsup_{j,k,l\to\infty} \abs{T(f_k,\pi_j) - \frac{1}{n_l\cdot l}\sum_{i=1}^{n_l} [[\gamma^{\delta_l}_{i,l}]](f,\pi)} = 0
  .
\end{equation*}
The first term does not depend on $l$ and the second term does not depend on $j,k$, so we conclude that
\begin{equation*}
  \lim_{j,k\to\infty} T(f_k,\pi_j) = \lim_{l\to\infty} \frac{1}{n_l\cdot l}\sum_{i=1}^{n_l} [[\gamma^{\delta_l}_{i,l}]](f,\pi)
\end{equation*}
and thus \eqref{TfpiTfkpij} gives
\begin{equation}\label{diagonalization_general}
  T(f,\pi) = \lim_{l\to\infty} \frac{1}{n_l\cdot l}\sum_{i=1}^{n_l} [[\gamma^{\delta_l}_{i,l}]](f,\pi) 
\end{equation}
for all $f\in\Lip(E,\mathbb{R}),\pi\in\Lipb(E,\mathbb{R})$, as claimed.

We next close the loops.
Let $\hat{\gamma}_{i,l}$ be the piecewise-geodesic closed curve formed by concatenating $\gamma^{\delta_l}_{i,l}$ with the geodesic $G_{e(\gamma_{i,l}), b(\gamma_{i,l})}$ connecting the endpoints of $\gamma_{i,l}$ in reverse order, and let $\overline{\gamma}_{i,l}=G_{b(\gamma_{i,l}), e(\gamma_{i,l})}$ denote the same geodesic traversed in the other direction.
Here we have used the fact that $\gamma^{\delta_l}_{i,l}$ and $\gamma_{i,l}$ have the same endpoints.
From \Cref{MorreyIsNorm} and \eqref{ReversedGeodesics} it follows that \eqref{diagonalization_general} can be rewritten as
\begin{align}\label{important_identity}
  T(f,\pi) = \lim_{l \to \infty} \left(\frac{1}{n_l\cdot l} \sum_{i=1}^{n_l} [[ \hat{\gamma}_{i,l} ]](f,\pi) + \frac{1}{n_l\cdot l} \sum_{i=1}^{n_l} [[ \overline{\gamma}_{i,l} ]](f,\pi)\right)
\end{align}
for all $l\in\mathbb{N},f\in\Lip(E,\mathbb{R}),\pi\in\Lipb(E,\mathbb{R})$.
As we shall demonstrate, the last sum becomes negligible in the limit $l\to\infty$, and we can see this by bounding the lengths of the geodesics $\overline{\gamma}_{i,l}$.

Since $\overline{\gamma}_{i,l}$ is a geodesic between the starting and ending points of $\gamma_{i,l}$, 
\begin{equation*}
  \length(\overline{\gamma}_{i,l})=d(b(\gamma_{i,l}), e(\gamma_{i,l}))
  .
\end{equation*}
We split according to whether the beginning and ending points are distant or not.
Recall the indicator functions $\phi_{r,b},\phi_{r,e}$ from \eqref{EndpointCutoffs}, where $r\in\mathbb{Q}\cap(0,\infty)$.
We have
\begin{equation*}
  \length(\overline{\gamma}_{i,l}) \leq 2r + l\phi_{r,b}(\gamma_{i,l}) +l\phi_{r,e}(\gamma_{i,l})
  .
\end{equation*}
To see this, note that if the beginning and ending points are both within distance $r$ of the fixed reference point $x_0$ then they are at most $2r$ apart, and otherwise they are separated by at most the length $\length(\gamma_{i,l})$, which is $l$ by \eqref{stronglawlength}.
In particular,
\begin{align*}
  &\limsup_{l\to\infty} \frac{1}{n_l\cdot l} \sum_{i=1}^{n_l} \length(\overline{\gamma}_{i,l}) 
  \\&\quad
  \leq \limsup_{l\to\infty} \left(\frac{2r}{l} + \frac{1}{n_l}\sum_{i=1}^{n_l} \phi_{r,b}(\gamma_{i,l}) + \frac{1}{n_l}\sum_{i=1}^{n_l} \phi_{r,e}(\gamma_{i,l})\right)
  \\&\quad
  = 2\mu_T(\set{x\in E\colon d(x,x_0)\geq r})
\end{align*}
by \eqref{diag_phi}.
The left-hand side does not depend on the choice of $r\in\mathbb{Q}\cap(0,\infty)$, so we may send $r\to\infty, r\in\mathbb{Q}$.
Since $\mu_T$ is a finite measure and the sets $\set{x\in E\colon d(x,x_0)\geq r}$ are decreasing in $r$ with empty intersection, continuity of measure gives $\mu_T(\set{x\in E\colon d(x,x_0)\geq r})\to 0$ as $r\to\infty$.
We conclude that
\begin{equation}\label{ExtraLengthSmall}
  \limsup_{l\to\infty} \frac{1}{n_l\cdot l} \sum_{i=1}^{n_l} \length(\overline{\gamma}_{i,l}) = 0
  .
\end{equation}

The rest of the proof follows easily.
By \Cref{easyestimateformass},
\begin{align*}
  \limsup_{l\to\infty} \abs{\frac{1}{n_l\cdot l} \sum_{i=1}^{n_l} [[ \overline{\gamma}_{i,l} ]](f,\pi)} 
  \leq 
  \norm{f}_\infty \Lip(\pi) \limsup_{l\to\infty} \frac{1}{n_l\cdot l} \sum_{i=1}^{n_l} \length(\overline{\gamma}_{i,l}) = 0
\end{align*}
so \eqref{closed1} follows from \eqref{important_identity}.

We next establish \eqref{closed2}.  
For the upper bound, we observe that the inequalities
\begin{align*}
  \mathbb{M}(\hat{\gamma}_{i,l})&\leq \length(\hat{\gamma}_{i,l}) = \length(\gamma^{\delta_l}_{i,l}) + \length(\overline{\gamma}_{i,l}) \leq \length(\gamma_{i,l}) + \length(\overline{\gamma}_{i,l}) = l+\length(\overline{\gamma}_{i,l}) 
\end{align*}
and the convergence \eqref{ExtraLengthSmall} together imply
\begin{align*}
  \lim_{l\to\infty} \frac{1}{n_l\cdot l}\sum_{i=1}^{n_l} \mathbb{M}(\hat{\gamma}_{i,l}) \leq \lim_{l\to\infty} \frac{1}{n_l\cdot l}\sum_{i=1}^{n_l}\length(\hat{\gamma}_{i,l}) \leq 1.
\end{align*}
The matching lower bound follows from a lower semi-continuity argument.
For a $1$-current $T$,  Ambrosio and Kirchheim's \cite{AmbrosioKirchheim}*{Proposition 2.7} asserts that the mass of $T$ can be computed as 
\begin{align*}
  \mathbb{M}(T) = \sup \sum_{p=1}^\infty \left|T(f_p,\pi_p)\right| 
\end{align*}
where the supremum is over sequences $f_p \in \Lipb(E)$ and $\pi_p \in \Lip(E)$ such that $\Lip(\pi_p) \leq 1$ for all $p$ and $\sum_p |f_p| \leq 1$.  
Suppose $f_p, \pi_p$ satisfy these assumptions.  
The identity \eqref{diagonalization_general} states that
\begin{align*}
 \left|T(f_p,\pi_p)\right|  = \left|\lim_{l \to \infty}  \frac{1}{n_l\cdot l}\sum_{i=1}^{n_l} [[\hat{\gamma}_{i,l}]](f_p,\pi_p) \right|
 .
\end{align*}
Applying Fatou's lemma for series (i.e., for the $\sigma$-finite measure space $\mathbb{N}$ under counting measure) and \Cref{easyestimateformass} yields
\begin{align*}
 \sum_{p=1}^\infty \left|T(f_p,\pi_p)\right|  &=  \sum_{p=1}^\infty \left|\lim_{l \to \infty} \frac{1}{n\cdot l}\sum_{i=1}^{n} [[\hat{\gamma}_{i,l}]](f_p,\pi_p) \right|\\
 &\leq  \liminf_{l \to \infty} \frac{1}{n_l\cdot l}\sum_{i=1}^{n_l}  \sum_{p=1}^\infty \left|[[\hat{\gamma}_{i,l}]](f_p,\pi_p) \right| \\
 &\leq \liminf_{l \to \infty}  \frac{1}{n\cdot l} \sum_{i=1}^{n}\mathbb{M}([[\hat{\gamma}_{i,l}]])\\
 &\leq 1
 .
\end{align*}
Taking the supremum over $f_p,\pi_p$ and recalling that $\mathbb{M}(T)=1$ by assumption,
\begin{align*}
  1 = \mathbb{M}(T) \leq \liminf_{l \to \infty}  \frac{1}{n_l\cdot l} \sum_{i=1}^{n_l}\mathbb{M}([[\hat{\gamma}_{i,l}]]) \leq 1,
\end{align*} 
which completes the proof of \eqref{closed2}.  

Finally, note that the uniform upper bound $\mathbb{M}(\hat{\gamma}_{i,l})\leq \length(\hat{\gamma}_{i,l}) \leq 2l$ follows from $\length(\gamma^{\delta_l}_{i,l}) \leq \length(\gamma_{i,l}) = l$ and the fact that concatenating a path with a geodesic between its endpoints can at most double the length.
\end{proof}

\section{Proof of the main result}\label{ProofOfMainResult}

Finally, we use \Cref{surgery-eta} and \Cref{bbassertion} to prove our main result.

\begin{proof}[Proof of \Cref{approximationMetric}]
Fix $T$ and $0<\epsilon<1$.
The case $T=0$ is trivial since we may take $\lambda_{i,n}=0$ for all $i,n$.
Otherwise, let $(n_l)_{l\in\mathbb{N}}$ and $(\hat{\gamma}_{j,l})_{l\in\mathbb{N},1\leq j\leq n_l}$ be the positive integers and closed piecewise-geodesic curves guaranteed by \Cref{bbassertion}.
Apply \Cref{surgery-eta} to each $\hat{\gamma}_{j,l}$ with $\eta=\tfrac{1}{2}\epsilon$ to produce positive integers $\tilde{N}_{j,l}$ and finite sequences $(\tilde{\gamma}_{j,l,k})_{1\leq k\leq N_{j,l}}$ of closed piecewise-geodesic curves.
For each $l\in\mathbb{N}$, set $N_l = \sum_{j=1}^{n_l} \hat{N}_{j,l}$ and let $(\bar{\gamma}_{i,l})_{1\leq i\leq N_l}$ be the sequence of length $N_l$ containing the curves $\tilde{\gamma}_{j,l,k}$ for $1\leq j\leq n_l$, $1\leq k\leq N_{j,l}$, say in lexicographic order.
Then by construction and by \Cref{surgery-eta}\ref{item:SurgeryMeasureIdentity-eta}
\begin{equation}\label{Tbargamma}
  T(\omega) = \lim_{l\to\infty} \frac{\mathbb{M}(T)}{n_l \cdot l} \sum_{j=1}^{n_l}  [[\hat{\gamma}_{j,l} ]](\omega) 
  = \lim_{l\to\infty} \frac{\mathbb{M}(T)}{n_l \cdot l} \sum_{i=1}^{N_l}  [[\bar{\gamma}_{i,l} ]](\omega) 
  .
\end{equation}
Similarly, \Cref{surgery-eta}\ref{item:SurgeryTotalLength-eta}--\ref{item:SurgeryBallGrowth-eta} give
\begin{gather}
  \label{lengthlimsup}
  \limsup_{l\to\infty} \frac{1}{n_l\cdot l}\sum_{i=1}^{N_l} \length(\bar{\gamma}_{i,l})
  \leq (1+\tfrac{1}{2}\epsilon) \lim_{l\to\infty} \frac{1}{n_l\cdot l}\sum_{j=1}^{n_l} \length(\hat{\gamma}_{j,l}) 
  = 1+\tfrac{1}{2}\epsilon
  ,
  \\
  \label{MorreyBounded}
  \Morrey{\bar{\gamma}_{i,l}}\leq 4C'/\epsilon^2 \quad\text{for all $i,l$}
  .
\end{gather}

By construction, $N_l\geq n_l$, and hence $N_l\to\infty$ as $l\to\infty$.
Along with \eqref{lengthlimsup}, this shows that we can choose a subsequence $(l_r)_{r\in\mathbb{N}}$ such that 
\begin{equation}\label{ScaledLengthSum}
  \frac{1}{n_{l_r}\cdot l_r}\sum_{i=1}^{N_{l_r}} \length(\bar{\gamma}_{i,{l_r}}) \leq 1+\epsilon
  \quad\text{for all }r\in\mathbb{N}
\end{equation}
and such that $(N_{l_r})_{r\in\mathbb{N}}$ is strictly increasing.

With these preparations, we can now construct the curves $\gamma_{i,n}$ and scalars $\lambda_{i,n}$, $1\leq i\leq n$.
If $n=N_{l_r}$ for some (necessarily unique) $r$, then set 
\begin{equation*}
  \gamma_{i,n} = \bar{\gamma}_{i,l_r}
  , 
  \quad
  \lambda_{i,n} = \frac{\mathbb{M}(T) \cdot \length(\bar{\gamma}_{i,l_r})}{n_{l_r}\cdot l_r}
  ,
  \quad 1\leq i\leq N_{l_r}
  .
\end{equation*}
If $N_{l_r}<n<N_{l_{r+1}}$, extend the sequences in a trivial way by setting $(\lambda_{i,n}, \gamma_{i,n})=(\lambda_{i,{N_{l_r}}}, \gamma_{i,{N_{l_r}}})$ for $i\leq N_{l_r}$ and $\lambda_{i,n}=0$ for $i>N_{l_r}$.
Similarly, if $n<N_{l_1}$ then set $\lambda_{i,n}=0$ for all $i$.
In the latter two cases with $\lambda_{i,n}=0$, the curves $\gamma_{i,n}$ are irrelevant; for definiteness, they can be taken to be curves of zero length, so that $\Morrey{\gamma_{i,n}}\leq 4C'/\epsilon^2$ holds trivially.

If $N_{l_r}\leq n<N_{l_{r+1}}$ then 
\begin{align*}
  \sum_{i=1}^n |\lambda_{i,n}| = \sum_{i=1}^{N_{l_r}} \lambda_{i,N_{l_r}} = \sum_{i=1}^{N_{l_r}} \frac{\mathbb{M}(T) \cdot \length(\bar{\gamma}_{i,l,r})}{n_{l_r}\cdot l_r} \leq (1+\epsilon)\mathbb{M}(T)
\end{align*}
by \eqref{ScaledLengthSum}.
Similarly 
\begin{align*}
  \lim_{n \to \infty} \sum_{i=1}^n \lambda_{i,n} \frac{[[\gamma_{i,n}]](\omega)}{\length(\gamma_{i,n})} 
  = \lim_{r \to \infty} \sum_{i=1}^{N_{l_r}} \frac{\mathbb{M}(T)}{n_{l_r}\cdot l_r} [[\gamma_{i,n}]](\omega)
  = T(\omega)
\end{align*}
by taking the limit \eqref{Tbargamma} along a subsequence.
Finally \eqref{MorreyBounded} shows that the curves $\gamma_{i,n}$ have uniformly bounded Morrey norms, and in \Cref{approximationMetric} we may take $C=4C'$ where $C'$ is the universal constant from \Cref{surgery-eta}.
\end{proof}

\appendix

\section{Separability and Continuity Results for Lipschitz Functions}\label{appendix}

\begin{lemma}\label{Lipschitz_separable}
  Let $E$ be a complete separable metric space.  
  The space of real-valued Lipschitz functions on $E$, $\Lip(E,\mathbb{R})$, is separable in the topology of pointwise convergence and convergence of Lipschitz constants.
  That is, there exists a countable set $\{\pi_n\}_{n \in \mathbb{N}} \subset \Lipb(E,\mathbb{R})$ such that for every $\pi \in \Lip(E,\mathbb{R})$ there exists a subsequence $\{\pi_{n_k}\}_{k \in \mathbb{N}} \subset \{\pi_n\}_{n \in \mathbb{N}} $ such that
  \begin{gather*}
    \text{for every $x \in E$,}
    \quad
    \pi_{n_k}(x) \to \pi(x)
    \quad
    \text{as $k \to \infty$, and}
    \\
    \Lip(\pi_{n_k}) \to \Lip(\pi)
    \quad\text{as $k \to \infty$}
    .
  \end{gather*}
  Furthermore we may choose the subsequence so that $\Lip(\pi_{n_k}) \leq \Lip(\pi)$ for all $k$.
  If in addition $\pi \in \Lipb(E,\mathbb{R})$, then we may choose the subsequence such that
  \begin{align*}
    \norm{\pi_{n_k}}_\infty\leq\norm{\pi}_\infty
    ,
    \qquad
    \|\pi_{n_k}\|_\infty \to \|\pi\|_\infty
    \quad\text{as $k\to\infty$}
    .
  \end{align*}
\end{lemma}
\begin{proof}
If $E$ is finite then $\Lip(E,\mathbb{R})$ is a finite-dimensional vector space and the result is trivial.
Otherwise, let $E_\mathbb{Q}$ denote a countably infinite dense subset of $E$, enumerated as $E_\mathbb{Q}=\set{e_j}_{j \in \mathbb{N}}$ with $e_j\neq e_{j'}$ for all $j\neq j'$.

Write $\tau_M(t)=\max\{-M, \min\{t,M\}\}$ for the signed truncation of a real number by the number $M$.
The desired family $\set{\pi_n}_{n\in\mathbb{N}}$ consists of functions of the form
\begin{equation*}
  g_{l;c_1,\dotsc,c_l;M,C}(x)=\tau_M \circ \left( \max_{1\leq j\leq l} (c_j - C d(x,e_j)) \right)
\end{equation*}
over the countable collection of indices $l\in\mathbb{N}$, $(c_1,\dotsc,c_l)\in\mathbb{Q}^l$, $M,C\in\mathbb{Q}\cap\cointerval{0,\infty}$.
By construction, $g_{l;c_1,\dotsc,c_l;M,C}\in\Lipb(E,\mathbb{R})$ with $\norm{g_{l;c_1,\dotsc,c_l;M,C}}_\infty\leq M$ and $\Lip(g_{l;c_1,\dotsc,c_l;M,C})\leq C$.

We remark that if the metric $d$ is bounded and the values $\pi(e_1),\dotsc,\pi(e_l)$ and $\Lip(\pi)$ are rational, then by choosing $c_j=\pi(e_j)$, $C=\Lip(\pi)$, and $M$ sufficiently large, the function $g_{l;c_1,\dotsc,c_l;M,C}$ coincides with the McShane-Whitney Lipschitz extension \cites{McShane, Whitney} of the restriction $\pi\vert_{\set{e_1,\dotsc,e_l}}$.
As we next explain, the general case can be handled specifying suitably close rational approximations.

Let $\pi\in\Lip(E,\mathbb{R})$ be given.
If $\norm{\pi}_\infty=0$ or $\Lip(\pi)=0$ then $\pi$ is constant, say $\pi(x)=y$ for all $x\in E$, and the result is trivial: we may choose $y_k\in\mathbb{Q}$ such that $\abs{y_k}\leq\abs{y}$ and $y_k\to y$, and then $\pi_{n_k}=\pi_{1;y_k;\abs{y_k},0}$ gives the desired sequence.

Otherwise, fix a sequence $(\epsilon_k)_{k\in\mathbb{N}}$ with $0<\epsilon_k<1$, $\epsilon_k\to 0$ as $k\to\infty$.
Define
\begin{gather*}
  r_k=\min\set{1,\min\set{d(e_j,e_{j'})\colon j,j'\in\set{1,\dotsc,k}, j\neq j'}}
  ,
  \\
  \delta_k = \epsilon_k \min\set{\norm{\pi}_\infty, \tfrac{1}{4}r_k\Lip(\pi)}
  ;
\end{gather*}
note that $0<r_k<\infty$ and $\delta_k>0$ for all $k$, and $\delta_k\to 0$ as $k\to\infty$ even if $\pi$ is unbounded.
Then we may choose $C_k\in\mathbb{Q}\cap\cointerval{0,\infty}$ and $c_{j,k}\in\mathbb{Q}$, $j\in\set{1,\dotsc,k}$, such that
\begin{gather*}
  (1-\tfrac{1}{2}\epsilon_k)\Lip(\pi) \leq C_k \leq \Lip(\pi)
  ,
  \\
  \abs{c_{j,k} - (1-\epsilon_k)\pi(e_j)} \leq \delta_k
  ,
\end{gather*}
and set $M_k=\max\set{\abs{c_{j,k}}\colon j\in\set{1,\dotsc,k}}$ and $\pi_{n_k}=g_{k;c_{1,k},\dotsc,c_{k,k};M_k,C_k}$.
By construction, $\Lip(\pi_{n_k})\leq C_k\leq\Lip(\pi)$, as claimed.
If $\pi$ is bounded, then from $\abs{c_{j,k}}\leq (1-\epsilon_k)\pi(e_j)+\delta_k\leq (1-\epsilon_k)\norm{\pi}_\infty+\epsilon_k\norm{\pi}_\infty$ we obtain $M_k\leq \norm{\pi}_\infty$ and hence $\norm{\pi_{n,k}}_\infty\leq\norm{\pi}_\infty$, also as claimed.

We next show that $\pi_{n_k}(e_{j'})=c_{j',k}$ for $j'\in\set{1,\dotsc,k}$.
To this end note that for $j,j'\in\set{1,\dotsc,k}$ distinct,
\begin{align*}
  \abs{c_{j',k}-c_{j,k}} &\leq 2\delta_k + (1-\epsilon_k)\abs{\pi(e_{j'})-\pi(e_j)}
  \\&
  \leq \tfrac{1}{2}\epsilon_k r_k\Lip(\pi) + (1-\epsilon_k)\Lip(\pi)d(e_{j'},e_j)
  \\&
  \leq \tfrac{1}{2}\epsilon_k d(e_{j'},e_j)\Lip(\pi) + (1-\epsilon_k)\Lip(\pi)d(e_{j'},e_j)
  \\&
  = (1-\tfrac{1}{2}\epsilon_k)\Lip(\pi)d(e_{j'},e_j) 
  \\&
  \leq C_k d(e_{j'},e_j)
\end{align*}
by the choice of $C_k$.
In particular, $c_{j',k}=c_{j',k}-C_k d(e_{j'},e_{j'})\geq c_{j,k}-C_k d(e_{j'},e_j)$, so that the maximum defining $g_{k;c_{1,k},\dotsc,c_{k,k};M_k,C_k}(e_{j'})$ is attained at $j=j'$.
Thus 
\begin{equation*}
  \pi_{n_k}(e_{j'})=g_{k;c_{1,k},\dotsc,c_{k,k};M_k,C_k}(e_{j'})=c_{j',k}
  ,
  \quad
  j'\in\set{1,\dotsc,k}
  ,
\end{equation*}
as claimed.

By construction, $\lim_{k\to\infty, k\geq j}c_{j,k}=\pi(e_j)$ for all $j$, and hence $\lim_{k\to\infty}\pi_{n_k}(e_j)=\pi(e_j)$ for all $j$.
For other $x\in E$, the general identity
\begin{align}\label{density_replacement}
  \abs{\pi(x) - \pi'(x)} \leq \inf_{q\in Q} \left( \abs{\pi(q) - \pi'(q)} + (\Lip(\pi)+\Lip(\pi'))d(x,q) \right),
\end{align}
applied with $Q=\set{e_1,\dotsc,e_l}$ gives
\begin{equation*}
  \limsup_{k\to\infty} \abs{\pi(x)-\pi_{n_k}(x)} \leq 2\Lip(\pi) \min_{1\leq j\leq l} d(x,e_j)
  .
\end{equation*}
The left-hand side does not depend on $l$, so by taking $l\to\infty$ and using the density of $E_\mathbb{Q}$ we find $\lim_{k\to\infty} \pi_{n_k}(x)=\pi(x)$.

Finally, we have already shown $\Lip(\pi_{n_k})\leq\Lip(\pi), \norm{\pi_{n_k}}_\infty\leq\norm{\pi}_\infty$.
The reverse direction follows easily from pointwise convergence: take $x,y\in E$ distinct and note
\begin{equation*}
  \liminf_{k\to\infty}\Lip(\pi_{n_k}) \geq \liminf_{k\to\infty} \frac{\abs{\pi_{n_k}(x)-\pi_{n_k}(y)}}{d(x,y)} = \frac{\abs{\pi(x)-\pi(y)}}{d(x,y)}
  .
\end{equation*}
The left-hand side does not depend on $x,y$, so taking the supremum over $x\neq y$ gives $\Lip(\pi)\leq \liminf_{k\to\infty}\Lip(\pi_{n_k})\leq\limsup_{k\to\infty}\Lip(\pi_{n_k})\leq\Lip(\pi)$.
A similar argument shows $\norm{\pi_{n_k}}_\infty\to\norm{\pi}_\infty$.
\end{proof}

\begin{lemma}\label{pi_bounds}
For $\gamma \in \Theta(E)$, define
\begin{align*}
\psi_{m,Q,\epsilon,C}(\gamma) &:= C\sum_{i=1}^m \min \left\{ 2C \frac{\length(\gamma)}{m}, 2\epsilon+ 2C\left(d(\gamma(s_i),Q) + d(\gamma(s_{i-1}),Q)\right)  \right\}\\
  &\qquad+2C^2\frac{\length(\gamma)^2}{m}
  ,\\
  \tilde{\psi}_{m,Q,\epsilon,C}(\gamma)&:= 2C^2\frac{\length(\gamma)^2}{m} + C\frac{\length(\gamma)}{m} \sum_{i=1}^m \min\set{2C, \epsilon+ 2Cd(\gamma(s_i),Q)}
  ,
\end{align*}
where $m \in \mathbb{N}$, $s_i=i\length(\gamma)/m$ for $i=0,\ldots,m$, $\epsilon>0$, $C>0$, and $Q \subset E_\mathbb{Q}$ is a finite set.  Then
\begin{align}
\abs{ \big. [[\gamma]](f,\pi-\pi') } &\leq \psi_{m,Q,\epsilon,C}(\gamma) \label{pi_estimate}\\
\abs{ \big. [[\gamma]](f-f',\pi) } &\leq \tilde{\psi}_{m,Q,\epsilon,C}(\gamma) \label{f_estimate}
\end{align}
for all $f, f' \in \Lipb(E,\mathbb{R})$, $\pi,\pi' \in \Lip(E,\mathbb{R})$ such that
\begin{align}
\norm{f}_\infty, \norm{f'}_\infty, \Lip(f), \Lip(f') &\leq C \label{f_assumptions}\\
\Lip(\pi),\Lip(\pi') &\leq C \label{pi_assumptions}\\
\sup_{q \in Q} \abs{\pi(q) - \pi'(q)} &\leq \epsilon,\label{pi_on_Q}\\
\sup_{q \in Q} \abs{f(q) - f'(q)} &\leq \epsilon. \label{f_on_Q}
\end{align}
\end{lemma}
\begin{proof}
For any $s_i^*\in[s_{i-1},s_i]$, we have
\begin{align}
  [[\gamma]](f,\pi) &= \int_0^{\length(\gamma)} f(\gamma(t)) \, d(\pi\circ\gamma)(t)
  \notag\\&
  = \sum_{i=1}^m f(\gamma(s_i^*)) (\pi(\gamma(s_i))-\pi(\gamma(s_{i-1}))) 
  \notag\\&\quad
  + \sum_{i=1}^m \int_{s_{l-1}}^{s_i} (f(\gamma(t))-f(\gamma(s_i^*))) \, d(\pi\circ\gamma)(t)
  .
  \label{RiemannStieltjesSum}
\end{align}
Since $\pi$ is parametrized by arc length it has Lipschitz constant bounded by 1, so that $\Lip(f\circ\gamma)\leq\Lip(f)$ and $\Lip(\pi\circ\gamma)\leq\Lip(\pi)$.
Hence each term in the second sum has absolute value bounded by $\Lip(f)\Lip(\pi)(s_i-s_{i-1})^2$, giving
\begin{equation}\label{mSumBound}
  \abs{\sum_{i=1}^m \int_{s_{l-1}}^{s_i} (f(\gamma(t))-f(\gamma(s_i^*))) \, d(\pi\circ\gamma)(t)} \leq \Lip(f)\Lip(\pi)\frac{\length(\gamma)^2}{m}
  .
\end{equation}

We first establish \eqref{pi_estimate}.
Replacing $\pi$ by $\pi-\pi'$ in \eqref{RiemannStieltjesSum}--\eqref{mSumBound} and noting that $\Lip(\pi-\pi')\leq\Lip(\pi)+\Lip(\pi')$,
\begin{align}
\abs{ \big. [[\gamma]](f,\pi-\pi') } &\leq \norm{f}_\infty \sum_{i=1}^m \abs{(\pi(\gamma(s_i))-\pi'(\gamma(s_i)))-(\pi(\gamma(s_{i-1}))-\pi'(\gamma(s_{i-1})))} \notag \\
&\quad +  \Lip(f) (\Lip(\pi)+\Lip(\pi'))\frac{\length(\gamma)^2}{m}\label{pi_minus_piprime}
.
\end{align}
For the summands in the right-hand side of \eqref{pi_minus_piprime}, we use two different bounds.
Under assumptions \eqref{pi_assumptions}--\eqref{pi_on_Q}, we may apply the bound \eqref{density_replacement} twice to obtain
\begin{align*}
  &\abs{(\pi(x)-\pi'(x)) - (\pi(y)-\pi'(y))} 
  \\&\quad
  \leq \inf_{q\in Q} \left( \abs{\pi(q) - \pi'(q)} + (\Lip(\pi)+\Lip(\pi'))d(x,q) \right)
  \\&\qquad
  +\inf_{q\in Q} \left( \abs{\pi(q) - \pi'(q)} + (\Lip(\pi)+\Lip(\pi'))d(y,q)\right)
  \\&\quad
  \leq 2\epsilon + 2C\left(d(x,Q) + d(y,Q)\right)
  .
\end{align*}
Alternatively, the Lipschitz bounds for $\pi,\pi'$ separately give
\begin{align*}
  \abs{(\pi(x)-\pi(y)) - (\pi'(x)-\pi'(y))} &\leq (\Lip(\pi)+\Lip(\pi')) d(x,y)
  \leq 2C d(x,y)
  .
\end{align*}
Applying the minimum of these two upper bounds to \eqref{pi_minus_piprime} gives
\begin{align*}
  &\abs{ \big. [[\gamma]](f,\pi-\pi') } 
  \\&\quad
  \leq \norm{f}_\infty \sum_{i=1}^m \min \left\{ 2C d(\gamma(s_i),\gamma(s_{i-1})),
  2\epsilon + 2C\left(d(\gamma(s_i),Q)+ d(\gamma(s_{i-1}),Q)\right) \right\}\\
  &\qquad+\Lip(f) (\Lip(\pi)+\Lip(\pi'))\frac{\length(\gamma)^2}{m},
\end{align*}
and \eqref{pi_estimate} follows from a further application of assumptions \eqref{f_assumptions}--\eqref{pi_assumptions} and the upper bound $d(\gamma(s_i),\gamma(s_{i-1}))\leq s_i-s_{i-1}=\frac{\length(\gamma)}{m}$.

For \eqref{f_estimate}, take $s_i^*=s_i$ in \eqref{RiemannStieltjesSum}--\eqref{mSumBound} with $f$ replaced by $f-f'$, noting that $\Lip(f-f')\leq\Lip(f)+\Lip(f')$ and $\Lip(\pi\circ\gamma)\leq\Lip(\pi)$.
Then, under assumptions \eqref{f_assumptions}--\eqref{pi_assumptions},
\begin{align}
  \abs{ \big. [[\gamma]](f-f',\pi) } &\leq \sum_{i=1}^m \abs{f(\gamma(s_i))-f'(\gamma(s_i))} \Lip(\pi)(s_i-s_{i-1}) 
  \notag\\&\quad 
  +  (\Lip(f)+\Lip(f')) \Lip(\pi)\frac{\length(\gamma)^2}{m}
  \notag\\&
  \leq \frac{C\length(\gamma)}{m} \sum_{i=1}^m \abs{f(\gamma(s_i))-f'(\gamma(s_i))} + 2C^2\frac{\length(\gamma)^2}{m}
  .
  \label{f_minus_fprime}
\end{align}
Apply \eqref{density_replacement} to $f,f'$ and use assumptions \eqref{f_assumptions} and \eqref{f_on_Q} to obtain
\begin{align*}
  |f(\gamma(s_i))-f'(\gamma(s_i))| &\leq \inf_{q\in Q} \left( \abs{f(q) - f'(q)} + (\Lip(f)+\Lip(f'))d(\gamma(s_i),q) \right)\\
  &\leq   \epsilon+ 2Cd(\gamma(s_i),Q) 
  .
\end{align*}
Alternatively, the uniform bounds for $f,f'$ separately give
\begin{equation*}
  \abs{f(\gamma(s_i))-f'(\gamma(s_i))} \leq \norm{f}_\infty+\norm{f'}_\infty \leq 2C
  .
\end{equation*}
Applying the minimum of these two upper bounds to \eqref{f_minus_fprime} gives \eqref{f_estimate}.
\end{proof}

\begin{lemma}\label{psi_vanishes}
Let $\overline{\eta}_l$ be a finite Borel measure on $\Theta_l(E)$.  
For $\psi_{m,Q,\epsilon,C}$, $\tilde{\psi}_{m,Q,\epsilon,C}$ the functions defined in \Cref{pi_bounds}, 
\begin{align}\label{mdeltaq}
  \lim_{m \to \infty} \lim_{Q \nearrow E_{\mathbb{Q}}} \lim_{\epsilon \to 0}  \int_{\Theta_l(E)} \psi_{m,Q,\epsilon,C}(\gamma) \, d\overline{\eta}_l (\gamma) =0
\end{align}
and 
\begin{align}\label{deltaq}
  \lim_{m \to \infty} \lim_{Q \nearrow E_{\mathbb{Q}}} \lim_{\epsilon \to 0}  \int_{\Theta_l(E)} \tilde{\psi}_{m,Q,\epsilon,C}(\gamma) \, d\overline{\eta}_l (\gamma) =0
  .
\end{align}
\end{lemma}

\begin{proof}
Restricted to curves $\gamma$ of length at most $l$, both functions satisfy the uniform upper bound
\begin{equation*}
  \max\set{\psi_{m,Q,\epsilon,C}(\gamma),\tilde{\psi}_{m,Q,\epsilon,C}(\gamma)} \leq 2C^2 l+2C^2 l^2
  .
\end{equation*}
For each fixed $x\in E$, the fact that $E_\mathbb{Q}$ is a dense subset implies the pointwise convergence $d(x,Q)\to 0$ as $Q\nearrow E_\mathbb{Q}$, and it follows that for each fixed $\gamma\in\Theta_l(E)$ there is the pointwise convergence $\psi_{m,Q,\epsilon,C}(\gamma),\tilde{\psi}_{m,Q,\epsilon,C}(\gamma)\to 2C^2 \length(\gamma)^2/m$ as $\epsilon\to 0,Q\nearrow E_\mathbb{Q}$.
Since $\overline{\eta}_l$ is a finite measure, Lebesgue's dominated convergence theorem gives
\begin{align*}
  \lim_{Q \nearrow E_{\mathbb{Q}}} \lim_{\epsilon \to 0}  \int_{\Theta_l(E)} \psi_{m,Q,\epsilon,C}(\gamma) \, d\overline{\eta}_l (\gamma) &= \frac{2C^2 l^2}{m} \overline{\eta}_l(\Theta_l(E))
\end{align*}
and similarly for $\tilde{\psi}_{m,Q,\epsilon,C}(\gamma)$.
The further limit $m\to\infty$ yields 0, as claimed.
\end{proof}

\begin{lemma}\label{psiSampling}
  Let $\gamma\in\Theta(E)$, with piecewise-geodesic sampling $\gamma^\delta$ as in \Cref{sampling}.
  Then for all admissible choices of $m,Q,\epsilon,C$,
  \begin{equation*}
    \lim_{\delta\to 0} \psi_{m,Q,\epsilon,C}(\gamma^\delta)=\psi_{m,Q,\epsilon,C}(\gamma)
    ,\quad
    \lim_{\delta\to 0} \tilde{\psi}_{m,Q,\epsilon,C}(\gamma^\delta)=\psi_{m,Q,\epsilon,C}(\gamma)
    .
  \end{equation*}
\end{lemma}
\begin{proof}
The result follows by observing that $\gamma^\delta(i\length(\gamma^\delta)/m)\to\gamma(i\length(\gamma)/m)$ as $\delta\to 0$, for all $i\in\set{0,\dotsc,m}$, and that $\psi_{m,Q,\epsilon,C},\tilde{\psi}_{m,Q,\epsilon,C}$ can be expressed as continuous functions of these values.
\end{proof}

\begin{lemma}\label{etalCopieseta1}
  Let $h\colon \Theta_1(E)\to\mathbb{R}$ be a bounded measurable function, let $l\in\mathbb{N}$, and let $j\in\set{1,\dotsc,l}$.
  Then
  \begin{equation*}
    \int_{\Theta_l(E)} h(\gamma\vert_{[j-1,j]}) \, d\overline{\eta}_l(\gamma)
    = \int_{\Theta_1(E)} h(\gamma) \, d\overline{\eta}_1(\gamma)
    .
  \end{equation*}
\end{lemma}
\begin{proof}
Modulo differences in notation, this follows directly from \cite{PS}*{Proposition~4.2}.
For $\gamma\in C(\mathbb{R},E)$, write $\pi_{[a,b]}(\gamma)=\gamma\vert_{[a,b]}$ and write $g^+(\gamma)$ for the function $\tilde{\gamma}$ given by $\tilde{\gamma}(t)=\gamma(t+1)$.
By construction, $\overline{\eta}_l$ is the image measure $\pi_{[0,l]\#}\hat{\eta}$, where $\hat{\eta}$ is the measure given by \cite{PS}*{Proposition~4.2} satisfying $g^+_\#\hat{\eta} = \hat{\eta}$.
Noting that $\pi_{[j-1,j]}\circ\pi_{[0,l]}=\pi_{[j-1,j]}=\pi_{[0,1]}\circ \underbrace{g^+\circ\dotsb\circ g^+}_{j-1\text{ times}}$, we obtain
\begin{equation*}
  \pi_{[j-1,j]\#}\overline{\eta}_l =\pi_{[j-1,j]\#}\hat{\eta}
  = \pi_{[0,1]\#} g^+_\# \dotsb g^+_\# \hat{\eta} = \pi_{[0,1]\#}\hat{\eta}=\overline{\eta}_1
\end{equation*}
and the statement of the lemma is the same conclusion in integral form.
\end{proof}

\begin{corollary}\label{etaRepeats}
  For all admissible choices of $m,Q,\epsilon,C$ and all $l\in\mathbb{N}$,
  \begin{align*}
    \int_{\Theta_l(E)} \frac{1}{l}\psi_{m\cdot l,Q,\epsilon,C}(\gamma) \, d\overline{\eta}_l (\gamma) &= 
    \int_{\Theta_1(E)} \psi_{m,Q,\epsilon,C}(\gamma) \, d\overline{\eta}_1 (\gamma)
    \\
    \int_{\Theta_l(E)} \frac{1}{l}\tilde{\psi}_{m\cdot l,Q,\epsilon,C}(\gamma) \, d\overline{\eta}_l (\gamma) &= 
    \int_{\Theta_1(E)} \tilde{\psi}_{m,Q,\epsilon,C}(\gamma) \, d\overline{\eta}_1 (\gamma)
  \end{align*}
\end{corollary}

\begin{proof}
Let $\gamma\in\Theta_l(E)$.
Recall the definition of $\psi_{m\cdot l,Q,\epsilon,C}(\gamma)$ from \Cref{pi_bounds}, with the points $s_i=il/(m\cdot l)=i/m$, $i=0,\dotsc,m\cdot l$.
Changing the indices of summation $i\in\set{1,\dotsc,m\cdot l}$ by $i=(j-1)m+k$ with $j\in\set{1,\dotsc,l}$ and $k\in\set{1,\dotsc,m}$,
\begin{align*}
  &\psi_{m\cdot l,Q,\epsilon,C}(\gamma) 
  \\&\quad
  = 
  C \sum_{i=1}^{m\cdot l} \min \left\{ 2C \frac{l}{m\cdot l}, 2\epsilon+ 2C\left(d(\gamma(i/m),Q) + d(\gamma((i-1)/m),Q)\right)  \right\}
  \\
  &\qquad
  +2C^2\frac{l^2}{m\cdot l}
  \\&
  = \sum_{j=1}^l \left[ C \sum_{k=1}^m \min \left\{ \frac{2C}{m}, 2\epsilon+ 2C\left(d(\gamma(j-1+\tfrac{k}{m}),Q) + d(\gamma(j-1+\tfrac{k-1}{m}),Q)\right)  \right\}
  \right.
  \\
  &\left.
  \qquad
  +\frac{2C^2}{m} \right]
  \\&
  =\sum_{j=1}^l \psi_{m,Q,\epsilon,C}(\gamma\vert_{[j-1,j]})
  .
\end{align*}
The result therefore follows from \Cref{etalCopieseta1}.
\end{proof}

\section*{Acknowledgements}
Y.-W. Chen is supported by the National Science and Technology Council of Taiwan under research grant number 113-2811-M-002-027.
J.~Goodman is supported by the Marsden Fund grants 20-UOO-079, 22-UOA-052, and 23-UOA-148, administered by the Royal Society Te Ap\=arangi, New Zealand.
D.~Spector is supported by the National Science and Technology Council of Taiwan under research grant numbers 110-2115-M-003-020-MY3/113-2115-M-003-017-MY3 and the Taiwan Ministry of Education under the Yushan Fellow Program.


\bibliographystyle{amsplain}

\end{document}